\documentclass[leqn,10pt]{amsart}
\usepackage[utf8]{inputenc}
\usepackage{a4}
\usepackage[safeinputenc=true,style=alphabetic,firstinits,maxnames=4,useprefix=true,backend=bibtex]{biblatex}

\bibliography{zdbrudnobib}
\usepackage{amsmath,amssymb,amscd,bbm}
\usepackage{enumerate}
\usepackage[mathscr]{eucal}
\usepackage[all]{xy}

\newtheorem{thm}{Theorem}[section]
\newtheorem{lemma}[thm]{Lemma}
\newtheorem{prop}[thm]{Proposition}
\newtheorem{cor}[thm]{Corollary}

\theoremstyle{definition}

\newtheorem{exa}[thm]{Example}

\theoremstyle{remark}

\numberwithin{equation}{section}

\newcommand{\vanish}[1]{\relax}       
\def\qedsymbol{\hbox to 1ex{\llap{\rule{0.25pt}{1ex}}\rlap{\rule{1ex}{0.25pt}}\lower0.25pt\rlap{\raise1ex\rlap{\rule{1ex}{0.25pt}}}\hskip1ex\llap{\rule{0.25pt}{1ex}}}}
\def\rlqed{\rlap{\rule{\hsize}{0pt}\kern-1ex\kern-1em\qed}} 

\newcounter{aufzi}
\newenvironment{aufzi}{\begin{list}{ {\upshape\alph{aufzi})}}{
        \usecounter{aufzi}
        \topsep1ex
        \parsep0cm
        \itemsep0.8ex
        \leftmargin1cm
        \labelwidth0.5cm
        \labelsep0.3cm
}}
{\end{list}}

\newcounter{aufzii}
\newenvironment{aufzii}{\begin{list}{\hfill {\upshape
(\roman{aufzii})}}{
        \usecounter{aufzii}
        \topsep1ex
        \parsep0cm
        \itemsep1ex
        \leftmargin1cm
        \labelwidth0.5cm
        \labelsep0.3cm
}}
{\end{list}}

\newcounter{aufziii}
\newenvironment{aufziii}{\begin{list}{ {\upshape\arabic{aufziii})}}{
        \usecounter{aufziii}
        \topsep1ex
        \parsep0cm
        \itemsep0.8ex
        \leftmargin1cm
        \labelwidth0.5cm
        \labelsep0.3cm
}}
{\end{list}}



\newcommand{\bbE}{\mathbb{E}}

\newcommand{\calB}{\mathcal{B}}

\newcommand{\calF}{\mathcal{F}}

\newcommand{\calZ}{\mathcal{Z}}

\def\bfX{\mathbf{X}}

\def\uK{\mathrm{K}}

\def\uN{\mathrm{N}}

\def\ub{\mathrm{b}}

\def\ue{\mathrm{e}}
\def\uf{\mathrm{f}}

\def\uk{\mathrm{k}}
\def\ul{\mathrm{l}}
\def\um{\mathrm{m}}
\def\un{\mathrm{n}}

\def\up{\mathrm{p}}

\def\ur{\mathrm{r}}
\def\us{\mathrm{s}}
\def\ut{\mathrm{t}}

\def\uw{\mathrm{w}}

\renewcommand{\ue}{\mathrm{e}}    
\def\id{\mathop{\mathrm{id}}\nolimits}   

\newcommand{\R}{\mathbb{R}}     
\newcommand{\N}{\mathbb{N}}
\newcommand{\Z}{\mathbb{Z}}

\newcommand{\comp}[1]{#1^\mathrm{c}}    
\DeclareMathOperator{\card}{card}  
\newcommand{\sdif}{\triangle}      

\newcommand{\Bigcap}[2][\relax]{%
 \ifx#1\relax \bigcap_{#2}
 \else \bigcap^{#1}_{#2}
 \fi}
\newcommand{\Bigcup}[2][\relax]{%
 \ifx#1\relax \bigcup_{#2}
 \else \bigcup^{#1}_{#2}
 \fi}

\def\fact#1#2{#1/#2}
\def\tfact#1#2{#1/#2}

\def\fact#1#2{{\raise0.2em\hbox{$#1$}\kern-0.2em/\kern-0.1em\lower0.2em\hbox{$#2$}}}
\def\tfact#1#2{{\raise0.1em\hbox{\small$#1$}\kern-0.1em/\kern-0.1em\lower0.1em\hbox{\small$#2$}}}

\newcommand{\norm}[2][\relax]{
   \ifx#1\relax \ensuremath{\left\Vert#2\right\Vert}
   \else \ensuremath{\left\Vert#2\right\Vert_{#1}}
   \fi}

\newcommand{\Bnorm}[2][\relax]{
   \ifx#1\relax \ensuremath{\Bigl\Vert#2\Bigr\Vert}
   \else \ensuremath{\Bigl\Vert#2\Bigr\Vert_{#1}}
   \fi}

\makeatletter
\newcommand{\tdprod}[2]{\ensuremath{%
  \setbox0=\hbox{\ensuremath{\langle#1,#2 \rangle}}
  \dimen@\ht0
  \advance\dimen@ by \dp0 (#1\rule[-\dp0]{0pt}{\dimen@}\,|#2\hspace{1pt})}}
\newcommand{\dprod}[2]{\ensuremath{%
  \setbox0=\hbox{\ensuremath{\left\langle#1,#2\right\rangle}}
  \dimen@\ht0
  \advance\dimen@ by \dp0 \left\langle\left.#1\rule[-\dp0]{0pt}{\dimen@}\,\right|#2\hspace{1pt}\right\rangle}}

\newcommand{\bdprod}[2]{\ensuremath{%
  \setbox0=\hbox{\ensuremath{\bigl\langle#1,#2\bigr\rangle}}
  \dimen@\ht0
  \advance\dimen@ by \dp0 \bigl\langle#1\bigl|\rule[-\dp0]{0pt}{\dimen@}\bigr.#2\hspace{1pt}\bigr\rangle}}
\newcommand{\Bdprod}[2]{\ensuremath{%
  \setbox0=\hbox{\ensuremath{\Bigl\langle#1,#2\Bigr\rangle}}
  \dimen@\ht0
  \advance\dimen@ by \dp0 \Bigl\langle#1\Bigl|\rule[-\dp0]{0pt}{\dimen@}\Bigr.#2\hspace{1pt}\Bigr\rangle}}
\newcommand{\tsprod}[2]{\ensuremath{%
  \setbox0=\hbox{\ensuremath{(#1,#2)}}
  \dimen@\ht0
  \advance\dimen@ by \dp0 (#1\rule[-\dp0]{0pt}{\dimen@}\,|#2\hspace{1pt})}}
\newcommand{\sprod}[2]{\ensuremath{%
  \setbox0=\hbox{\ensuremath{\left(#1,#2\right)}}
  \dimen@\ht0
  \advance\dimen@ by \dp0 \left(\left.#1\rule[-\dp0]{0pt}{\dimen@}\,\right|#2\hspace{1pt}\right)}}
\newcommand{\bsprod}[2]{\ensuremath{%
  \setbox0=\hbox{\ensuremath{\bigl(#1,#2\bigr)}}
  \dimen@\ht0
  \advance\dimen@ by \dp0 \bigl(#1\bigl|\rule[-\dp0]{0pt}{\dimen@}\bigr.#2\hspace{1pt}\bigr)}}
\newcommand{\Bsprod}[2]{\ensuremath{%
  \setbox0=\hbox{\ensuremath{\Bigl(#1,#2\Bigr)}}
  \dimen@\ht0
  \advance\dimen@ by \dp0 \Bigl(#1\Bigl|\rule[-\dp0]{0pt}{\dimen@}\Bigr.#2\hspace{1pt}\Bigr)}}
\makeatother

\newcommand{\Ell}[2][\relax]{
   \ifx#1\relax \mathrm{L}^{\mathrm{#2}}
   \else \mathrm{L}^{\mathrm{#2}}_{\mathrm{#1}}
   \fi}
\renewcommand{\Ell}[2][\relax]{
   \ifx#1\relax \mathrm{L}^{\!#2}
   \else \mathrm{L}^{\!#2}_{\mathrm{#1}}
   \fi}

\newcommand{\Wee}[2][\relax]{
   \ifx#1\relax \mathrm{W}^{\mathrm{#2}}
   \else \mathrm{W}^{\mathrm{#2}}_{\mathrm{#1}}
   \fi}
\newcommand{\Har}[2][\relax]{
   \ifx#1\relax \mathsf{H}^{\mathsf{#2}}
   \else   \mathsf{H}^{\mathsf{#2}}_{\mathrm{#1}}
   \fi}
\newcommand{\eM}{\mathrm{M}}     

\newcommand{\Probm}{\eM^1}

\def\prX{\mathrm X}    

\def\rlqed{\rlap{\rule{\hsize}{0pt}\kern-1ex\kern-1em\qed}}

\makeatletter

\def\maketag@@@@@#1{\llap{\hbox to\hsize{\m@th\normalfont#1}}%
\gdef\tagform@##1{\maketag@@@{(\ignorespaces##1\unskip\@@italiccorr)}}}

\def\eqtext#1{\gdef\tagform@##1{\maketag@@@@@{\ignorespaces##1\unskip\@@italiccorr\hfill}}\tag{#1}}%
\def\reqtext#1{\gdef\tagform@##1{\maketag@@@@@{\hfill\ignorespaces##1\unskip\@@italiccorr}}\tag{#1}}%
\def\leqtext#1{\gdef\tagform@##1{\maketag@@@@@{\ignorespaces##1\unskip\@@italiccorr}}\tag{#1}}%
\makeatother

\newcommand{\inv}{\mathrm{inv}}

\newcommand{\UT}[2]{\mathsf{UT}_{#1}(#2)}

\newcommand{\cntm}[1]{\left|#1\right|}

\newcommand{\kintl}[2]{\mathrm{int}_{#1}^{\mathrm l}(#2)}
\newcommand{\kbdrl}[2]{\partial_{#1}^{\mathrm l}(#2)}
\newcommand{\kintr}[2]{\mathrm{int}_{#1}^{\mathrm r}(#2)}
\newcommand{\kbdrr}[2]{\partial_{#1}^{\mathrm r}(#2)}

\newcommand{\avg}[1]{\bbE_{#1}}
\newcommand{\indi}[1]{\mathbf{1}_{#1}}

\newcommand{\kcc}[2]{\uK_{#1}^{0}(#2)}
\newcommand{\kcca}[2]{\overline{\uK}_{#1}^0(#2)}

\newcommand{\kc}[2]{\uK(#1,#2)}
\newcommand{\kca}[2]{\overline{\uK}(#1,#2)}
\newcommand{\kcl}[1]{\widehat{\uK}(#1)}

\newcommand{\aast}{A^{\ast}}
\newcommand{\aex}{A^!}
\newcommand{\adag}{A^{\dag}}

\newcommand{\del}{\mathrm{01}}
\newcommand{\ddel}{\mathrm{0110}}

\newcommand{\cin}[1]{\mathrm{I}(#1)}

\newcommand{\sym}[1]{\mathrm{Sym}_{#1}}
\date{\today}

\begin{document}

\title[Brudno Theorem]{Computable F{\o}lner monotilings \\ and a theorem of Brudno II.}

\author[Nikita Moriakov]{Nikita Moriakov}
\address{Delft Institute of Applied Mathematics, Delft University of Technology,
P.O. Box 5031, 2600 GA Delft, The Netherlands}

\email{n.moriakov@tudelft.nl}

\subjclass{Primary  37B10, 37B40, 03D15}
\renewcommand{\subjclassname}{\textup{2000} Mathematics Subject
    Classification}

\thanks{The author kindly acknowledges the support from ESA CICAT of TU Delft}
\date{\today}

\begin{abstract}
A theorem of A.A. Brudno \cite{brudno1982} says that the
Kolmogorov-Sinai entropy of a subshift $\bfX$ over $\N$ with respect
to an ergodic measure $\mu$ equals the asymptotic Kolmogorov
complexity of almost every word $\omega$ in $\bfX$. The purpose of
this article is to extend this result to subshifts over computable
groups that admit computable regular symmetric F{\o}lner monotilings, which we introduce in this work. These monotilings are a special type of computable F{\o}lner monotilings, which we defined earlier in \cite{moriakov1} in order to extend the initial results of Brudno \cite{brudno1974}. For every $d \in \N$, the groups $\Z^d$ and $\UT{d+1}{\Z}$ admit particularly nice computable regular symmetric F{\o}lner monotilings for which we can provide the required computing algorithms `explicitly'. 
\end{abstract}

\maketitle

\section{Introduction}
It was proved by A.A. Brudno in \cite{brudno1982} that the Kolmogorov-Sinai entropy of a $\N$-dynamical system equals a.e. the Kolmorogov complexity of its orbits. An important special case is the case of subshifts over $\N$. It says that if $\bfX=(\prX,\mu,\N)$ is an ergodic subshift over $\N$, then for $\mu$-a.e. $\omega \in \prX$ we have
\begin{equation*}
h(\bfX) = \limsup\limits_{n \to \infty} \frac{\uK(\omega|_{[1,\dots,n]})}{n},
\end{equation*}
where $h(\bfX)$ is the Kolmogorov-Sinai entropy of $\bfX$ and
$\uK(\omega|_{[1,\dots,n]})$ is the \emph{Kolmogorov complexity} of
the word $\omega|_{[1,\dots,n]}$ of length $n$. Roughly speaking,
$\uK(\omega|_{[1,\dots,n]})$ is the length of the shortest description
of $\omega|_{[1,\dots,n]}$ for an `optimal decompressor' that takes
finite binary words as the input and produces finite words as the
output. A similar result can be easily proven for ergodic subshifts
over the group $\Z$ of integers, but the question remains if one can
generalize this theorem beyond the $\Z$ and $\Z^d$ cases.

The purpose of this article is to extend this result of Brudno to
subshifts over computable groups that posses computable regular
symmetric F{\o}lner monotilings (see Section \ref{ss.compfoln} for the
definition). The class of all such groups includes, for every $d \in
\N$, the groups $\Z^d$ and, for every $d \geq 2$, the groups of
unipotent upper-triangular matrices $\UT{d}{\Z}$ with integer entries. Thus this gives a nontrivial abstract generalization of the result of S.G. Simpson in \cite{simpson2015} for $\Z^d$ case.

The article is structured as follows. We devote Section
\ref{ss.pamengroup} to the general preliminaries on amenable groups
and entropy theory. Regular F{\o}lner monotilings, which are a special
type of \emph{F{\o}lner monotilings} from the work \cite{weiss2001} of
B. Weiss, are introduced in Section \ref{ss.regfoln}. We provide some
basic notions from the theory of computability and Kolmogorov
complexity in Section \ref{ss.compcompl}, and in Section
\ref{ss.compspaces} we define computable spaces, word presheaves and
(asymptotic) Kolmogorov complexity of sections of these
presheaves. Section \ref{ss.compgroup}, based on the work
\cite{rabin1960}, contains the definition of a computable group and
some basic examples. We proceed by introducing computable F{\o}lner
monotilings in Section \ref{ss.compfoln} and explaining why the groups
$\Z^d$ and $\UT{d}{\Z}$ do admit computable regular symmetric
F{\o}lner monotilings. The main result of this article (Theorem \ref{thm.brudno}) is proved
in Section \ref{s.brudno}.

The article has some overlap with our previous article \cite{moriakov1}, where the original results \cite{brudno1974} of A.A. Brudno were extended. For instance, we use the notions of a computable space, a word presheaf and Kolmogorov complexity of sections of word presheaves. We provide this general setup in this draft too and do not use any results specific to our previous work. But there are significant distinctions as well. First of all, in this work we prove the equality of the Kolmogorov-Sinai entropy of a subshift $\bfX=(\prX,\mu,\Gamma)$ to the asymptotic Kolmogorov complexity of a word $\omega$ in $\prX$ for $\mu$-a.e. $\omega$. In the previous article we proved the equality of the topological entropy and the asymptotic Kolmogorov complexity of the word presheaf associated to $\bfX$. Techniques and ideas in the proofs of these theorems are very different, in particular, the algorithms that we use to encode words in this article are not related to the algorithms that we have given earlier. Secondly, the central results of this article (Theorem \ref{thm.brudno}) is proved under the assumption that the group $\Gamma$ admits a computable regular symmetric F{\o}lner monotiling, while in the previous work we imposed a weaker requirement that the group $\Gamma$ admits a computable F{\o}lner monotiling. Finally, in this article we rely on rather sophisticated tools from ergodic theory and entropy theory of amenable group actions, namely the work \cite{lindenstrauss2001} of E. Lindenstrauss and the work \cite{zorin2014} of P. Zorin-Kranich, which were not needed before. 

I would like to thank Alexander Shen
for reading the preprint and providing the feedback. I would like to thank Vitaly Bergelson, Eli Glasner and Benjamin Weiss
for the discussions on the topic. I would also like to thank my advisor Markus Haase for supporting this work.

\section{Preliminaries}

\subsection{Amenable groups and ergodic theory}
\label{ss.pamengroup}

In this section we will remind the reader of the classical notion of amenability, and state some results from ergodic theory of amenable group actions. We stress that all the groups that we consider are discrete and countably infinite. In what follows we shall rely mostly on \cite{zorin2014} and \cite{lindenstrauss2001}.

Let $\Gamma$ be a group with the counting measure $\cntm{\cdot}$. A sequence of finite sets $(F_n)_{n \geq 1}$ is called
\begin{aufziii}
\item a \textbf{left (right) weak F{\o}lner sequence} if for every finite set $K \subseteq \Gamma$ one has
\begin{equation*}
    \frac{\cntm{F_n \sdif K F_n}}{\cntm{F_n}} \to 0 \ \ \left( \text{resp. } \frac{\cntm{F_n \sdif F_n K}}{\cntm{F_n}} \to 0 \right);
\end{equation*}
\item a \textbf{left (right) strong F{\o}lner sequence} if for every finite set $K \subseteq \Gamma$ one has
\begin{equation*}
    \frac{\cntm{\kbdrl{K}{F_n}}}{\cntm{F_n}} \to 0 \ \ \left( \text{resp. }  \frac{\cntm{\kbdrr{K}{F_n}}}{\cntm{F_n}} \to 0 \right),
\end{equation*}
where
\begin{equation*}
\kbdrl{K}{F}:=K^{-1}F \cap K^{-1} \comp{F} \ \ \left( \text{resp. }  \kbdrr{K}{F}:=F K^{-1} \cap \comp{F} K^{-1} \right)
\end{equation*}
 is the \textbf{left (right) $K$-boundary} of $F$;

\item a \textbf{(C-)tempered sequence} if there is a constant $C$ such that for every $j$ one has $$\cntm{\bigcup\limits_{i<j} F_i^{-1} F_j} < C \cntm{F_j}.$$
\end{aufziii}

One can show a sequence of sets $(F_n)_{n \geq 1}$ is a weak left F{\o}lner sequence if and only if it is a strong left F{\o}lner sequence (see \cite{ceccherini2010}, Section 5.4), hence we will simply call it a left F{\o}lner sequence. The same holds for right F{\o}lner sequences. If we call a sequence of sets a `\textbf{F{\o}lner sequence}' without saying if it is `left' or `right', we always mean a left F{\o}lner sequence. A sequence of sets $(F_n)_{n \geq 1}$ which is simultaneously a left and a right F{\o}lner sequence is called a \textbf{two-sided F{\o}lner sequence}. A group $\Gamma$ is called \textbf{amenable} if it admits a left F{\o}lner sequence. It can be shown that every amenable group admits a two-sided F{\o}lner sequence. Since $\Gamma$ is infinite, for every F{\o}lner sequence $(F_n)_{n \geq 1}$ we have $\cntm{F_n} \to \infty$ as $n \to \infty$.

For finite sets $F, K \subseteq \Gamma$ the sets
\begin{equation*}
\kintl{K}{F}:=F \setminus \kbdrl{K}{F} \ \ \left( \text{resp. } \kintr{K}{F}:=F \setminus \kbdrr{K}{F}  \right)
\end{equation*}
are called the \textbf{left (right) $K$-interior} of $F$ respectively. It is clear that if a sequence of finite sets $(F_n)_{n \geq 1}$ is a left (right) F{\o}lner sequence, then for every finite $K \subseteq \Gamma$ one has \begin{equation*}
\cntm{\kintl{K}{F_n}} / \cntm{F_n} \to 1 \ \ \left( \text{resp. }  \cntm{\kintr{K}{F_n}}/\cntm{F_n} \to 1 \right)
\end{equation*}
as $n \to \infty$.

One of the reasons why F{\o}lner sequences are of interest in this
work is that they are `good' for averaging group actions. In what
follows all group actions are left actions. We denote the averages by
$\avg{g \in F}:=\frac{1}{\cntm{F}}\sum\limits_{g \in F}.$ The
following important theorem was proved by E. Lindenstrauss in \cite{lindenstrauss2001}.
\begin{thm}
Let $\bfX=(\prX,\mu,\Gamma)$ be a measure-preserving dynamical system, where the group $\Gamma$ is amenable and $(F_n)_{n \geq 1}$ is a tempered left F{\o}lner sequence. Then for every $f \in \Ell{1}{(\prX)}$ there is a $\Gamma$-invariant $\overline f \in \Ell{1}{(\prX)}$ such that
\begin{equation*}
\lim\limits_{n \to \infty} \avg{g \in F_n} f(g \omega) = \overline f(\omega)
\end{equation*}
for $\mu$-a.e. $\omega \in \prX$. If the system $\bfX$ is ergodic, then
\begin{equation*}
\lim\limits_{n \to \infty} \avg{g \in F_n} f(g \omega) = \int f d\mu
\end{equation*}
for $\mu$-a.e. $\omega \in \prX$.
\end{thm}

We will need a weighted variant of this result. A function $c$ on $\Gamma$ is called a \textbf{good weight} for pointwise convergence of ergodic averages along a tempered left F{\o}lner sequence $(F_n)_{n \geq 1}$ in $\Gamma$ if for every measure-preserving system $\bfX=(\prX,\mu,\Gamma)$ and every $f \in \Ell{\infty}{(\prX)}$ the averages
\begin{equation*}
\avg{g \in F_n} c(g) f(g \omega)
\end{equation*}
converge as $n \to \infty$ for $\mu$-a.e. $\omega \in \prX$.

We will use a special case of the Theorem 1.3 from \cite{zorin2014}.
\begin{thm}
\label{thm.retthm}
Let $\Gamma$ be a group with a tempered F{\o}lner sequence $(F_n)_{n \geq 1}$. Then for every ergodic measure-preserving system $\bfX=(\prX,\mu,\Gamma)$ and every $f \in \Ell{\infty}{(\prX)}$ there exists a full measure subset $\widetilde{X} \subseteq X$ such that for every $x \in \widetilde X$ the map $g \mapsto f(g x)$ is a good weight for the pointwise ergodic theorem along $(F_n)_{n \geq 1}$.
\end{thm}

We will now briefly remind the reader of the notion of Kolmogorov-Sinai entropy for amenable group actions. Let $\alpha=\{ A_1,\dots,A_n\}$ be a finite measurable partition of a probability space $(\prX,\calB, \mu)$. The function $\omega \mapsto \alpha(\omega)$, mapping a point $\omega \in \prX$ to the atom of the partition $\alpha$ containing $\omega$, is defined almost everywhere. The \textbf{information function} of $\alpha$ is defined as
\begin{equation*}
I_{\alpha}(\omega):=-\sum\limits_{i=1}^n \indi{A_i}(\omega) \log \mu(A_i)=-\log(\mu(\alpha(\omega))).
\end{equation*}
Then $I_{\alpha} \in \Ell{\infty}{(\prX)}$. The \textbf{Shannon entropy of a partition} $\alpha$ is defined by
\begin{equation*}
h_{\mu}(\alpha):=-\sum\limits_{i=1}^n \mu(A_i) \log(\mu(A_i)) = \int I_{\alpha} d\mu.
\end{equation*}
Entropy of a partition is always a nonnegative real number. If $\alpha,\beta$ are two finite measurable partitions of $\prX$, then
\begin{equation*}
\alpha \vee \beta:=\{ A \cap B: A \in \alpha, B \in \beta \}
\end{equation*}
is a finite measurable partition of $\prX$ as well. Given a measure-preserving dynamical system $\bfX=(\prX,\mu,\Gamma)$, where the discrete amenable group $\Gamma$ acts on $\prX$, we can also define (dynamical) entropy of a partition. First, for every element $g \in \Gamma$ and every finite measurable partition $\alpha$ we define a finite measurable partition $g^{-1} \alpha$ by
\begin{equation*}
g^{-1} \alpha = \{ g^{-1} A: A \in \alpha\}.
\end{equation*}
Next, for every finite subset $F \subseteq \Gamma$  and every partition $\alpha$ we define the partition
\begin{equation*}
\alpha^F:=\bigvee\limits_{g \in F} g^{-1} \alpha.
\end{equation*}
Let $(F_n)_{n \geq 1}$ be a F{\o}lner sequence in $\Gamma$ and $\alpha$ be a finite measurable partition of $\prX$. Then the limit
\begin{equation*}
h_{\mu}(\alpha,\Gamma):=\lim\limits_{n \to \infty} \frac{h_{\mu}(\alpha^{F_n})}{|F_n|}
\end{equation*}
exists, it is a nonnegative real number independent of the choice of a
F{\o}lner sequence due to the lemma of D.S. Ornstein and B. Weiss (see \cite{gromov1999},\cite{krieger2007}). The limit $h_{\mu}(\alpha,\Gamma)$ is called \textbf{dynamical entropy of $\alpha$}. We define the \textbf{Kolmogorov-Sinai entropy} of a measure-preserving system $\bfX=(\prX,\mu, \Gamma)$ by
\begin{equation*}
h(\bfX):=\sup\{ h_{\mu}(\alpha,\Gamma): \alpha \text{ a finite partition of } \prX\}.
\end{equation*}

We will need the Shannon-McMillan-Breiman theorem for amenable group actions. For the proof see \cite{lindenstrauss2001}.
\begin{thm}
\label{thm.smb}
Let $\bfX=(\prX,\mu,\Gamma)$ be an ergodic measure-preserving system
and $\alpha$ be a finite partition of $\prX$. Assume that $(F_n)_{n \geq 1}$ is a tempered F{\o}lner sequence in $\Gamma$ such that $\frac{|F_n|}{\log n} \to \infty$ as $n \to \infty$. Then there is a constant $h_{\mu}'(\alpha,\Gamma)$ s.t.
\begin{equation}
\label{eq.smb1}
\frac{I_{\alpha^{F_n}}(\omega)}{|F_n|} \to h_{\mu}'(\alpha,\Gamma)
\end{equation}
as $n \to \infty$ for $\mu$-a.e. $\omega \in \prX$ and in $\Ell{1}{(\prX)}$.
\end{thm}

Integrating both sides of the Equation \ref{eq.smb1} with respect to $\mu$, we deduce that
\begin{equation*}
\frac{h_{\mu}(\alpha^{F_n})}{|F_n|} \to h_{\mu}(\alpha,\Gamma) = h_{\mu}'(\alpha,\Gamma).
\end{equation*}
as $n \to \infty$. The Shannon-McMillan-Breiman theorem has the
following important corollary that will be used in the proof of
Theorem \ref{thm.brudnogeq} (\cite{glasner}, Corollary 14.36).
\begin{cor}
\label{cor.smb}
Let $\bfX = (\prX,\mu,\Gamma)$ be an ergodic measure-preserving system, $(F_n)_{n \geq 1}$
be a tempered F{\o}lner sequence in $\Gamma$ such that
$\frac{|F_n|}{\log n} \to \infty$ as $n \to \infty$. Let $\alpha$ is a
finite partition, then, given $\varepsilon>0$ and $\delta>0$, there
exists $n_0$ s.t. the following assertions hold:
\begin{aufzi}
\item For all $n \geq n_0$
\begin{equation*}
2^{-\cntm{F_n}(h_{\mu}(\alpha,\Gamma)+\varepsilon)} \leq \mu(A) \leq 2^{-\cntm{F_n}(h_{\mu}(\alpha,\Gamma)-\varepsilon)}
\end{equation*}
for all atoms $A \in \alpha^{F_n}$ with the exception of a set of atoms whose total measure is less than $\delta$.
\item For all $n \geq n_0$
\begin{equation*}
2^{-\cntm{F_n}(h_{\mu}(\alpha,\Gamma)+\varepsilon)} \leq \mu(\alpha^{F_n}(\omega)) \leq 2^{-\cntm{F_n}(h_{\mu}(\alpha,\Gamma)-\varepsilon)}
\end{equation*}
for all but at most $\delta$ fraction of elements $\omega \in \prX$.
\end{aufzi}
\end{cor}
\begin{proof}
By Therorem \ref{thm.smb}, $\frac{I_{\alpha^{F_n}}(\omega)}{|F_n|} \to h_{\mu}(\alpha,\Gamma)$ for a.e. $\omega$ and hence also in measure. Thus, given $\varepsilon,\delta>0$ as above, there is $n_0$ s.t. for all $n \geq n_0$ we have
\begin{equation*}
\mu\{ \omega \in \prX: \left| \frac{I_{\alpha^{F_n}}(\omega)}{|F_n|} - h_{\mu}(\alpha,\Gamma)\right|\geq \varepsilon \}< \delta.
\end{equation*}
It is now clear that both assertions follow.
\end{proof}

\subsection{(Regular) F{\o}lner monotilings}
\label{ss.regfoln}
The purpose of this section is to discuss the notion of a
\emph{F{\o}lner monotiling}, that was introduced by B. Weiss in \cite{weiss2001}. However, in this article we have to introduce both `left' and `right' monotilings, while the original notion introduced by Weiss is a `left' monotiling. The (new) notion of a \emph{regular F{\o}lner monotiling}, central to the results of this paper, will also be suggested below.

A \textbf{left monotiling} $[F, \calZ]$ in a discrete group $\Gamma$ is a pair of a finite set $F\subseteq \Gamma$, which we call \textbf{a tile}, and a set $\calZ \subseteq \Gamma$, which we call a set of \textbf{centers}, such that $\{ F z: z \in \calZ \}$ is a covering of $\Gamma$ by disjoint translates of $F$. Respectively, given a \textbf{right monotiling} $[\calZ,F]$ we require that $\{ z F: z \in \calZ \}$ is a covering of $\Gamma$ by disjoint translates of $F$. A \textbf{left (right) F{\o}lner monotiling} is a sequence of monotilings $([F_n,\calZ_n])_{n \geq 1}$ (resp. $([\calZ_n,F_n])_{n \geq 1}$) s.t. $(F_n)_{n \geq 1}$ is a left (resp. right) F{\o}lner sequence in $\Gamma$. A left F{\o}lner monotiling $([F_n,\calZ_n])_{n \geq 1}$ is called \textbf{symmetric} if for every $k \geq 1$ the set of centers $\calZ_k$  is symmetric, i.e. $\calZ_k^{-1} = \calZ_k$. It is clear that if $([F_n,\calZ_n])_{n \geq 1}$ is a symmetric F{\o}lner monotiling, then $([\calZ_n,F_n^{-1}])_{n \geq 1}$ is a right F{\o}lner monotiling.

We begin with a basic example.
\begin{exa}
\label{ex.zdmonotiling}
Consider the group $\Z^d$ for some $d \geq 1$ and the F{\o}lner
sequence $(F_n)_{n\geq 1}$ in $\Z^d$ given by $$F_n:=
[0,1,2,\dots,n-1]^d.$$ Furthermore, for every $n$ let $$\calZ_n:= n
\Z^d.$$ It is easy to see that $([F_n,\calZ_n])_{n \geq 1}$ is a
symmetric F{\o}lner monotiling of $\Z^d$, and that $(F_n)_{n \geq 1}$
is a tempered two-sided F{\o}lner sequence.
\end{exa}

A less trivial example is given by F{\o}lner monotilings of the
discrete Heisenberg group $\UT{3}{\Z}$. We will return to F{\o}lner monotilings of $\UT{d}{\Z}$ for $d > 3$ later.

\begin{exa}
\label{ex.hmonotiling}
Consider the group $\UT{3}{\Z}$, i.e. the discrete Heisenberg group
$H_3$. By the definition,
\begin{equation*}
\UT{3}{\Z} := \left\{ \left( \begin{matrix}
        1 & a & c \\
        0 & 1 & b \\
        0 & 0 & 1
      \end{matrix} \right): a,b,c \in \Z  \right\}.
\end{equation*}
To simplify the notation, we will denote a matrix
\begin{equation*}
\left( \begin{matrix}
        1 & a & c \\
        0 & 1 & b \\
        0 & 0 & 1
      \end{matrix} \right) \in \UT{3}{\Z}
\end{equation*}
by the corresponding triple $(a,b,c)$ of its
entries. Then the products and inverses in $\UT{3}{\Z}$ can be computed by the formulas
\begin{align*}
(a,b,c)(x,y,z) &= (a+x,b+y,c+z+ya), \\
(a,b,c)^{-1} &= (-a,-b,ba-c).
\end{align*}
For every $n \geq 1$, consider the subgroup 
\begin{equation*}
\calZ_n:=\{ (a,b,c) \in \UT{3}{\Z}: a,b \in n \Z, c \in n^2 \Z \}.
\end{equation*}
This is a finite index subgroup, and it is easy to see that for every $n$ the finite set
\begin{equation*}
F_n:=\{ (a,b,c) \in \UT{3}{\Z}: 0 \leq a,b < n , 0 \leq c < n^2  \}
\end{equation*}
is a fundamental domain for $\calZ_n$. One can show
(see \cite{lenz2011}) that $(F_n)_{n \geq 1}$ is a left F{\o}lner sequence, and a
similair argument shows that it is a right F{\o}lner sequence as
well. $([F_n,\calZ_n])_{n\geq 1}$ is a symmetric F{\o}lner monotiling.
In order to check temperedness of $(F_n)_{n \geq 1}$, note that for every $n > 1$
\begin{align*}
\bigcup\limits_{i<n} F_i^{-1} F_n \subseteq F_n^{-1} F_n,
\end{align*}
where 
\begin{equation*}
F_n^{-1} \subseteq \{ (a,b,c): -n < a,b \leq 0, -n^2<c<n^2\}.
\end{equation*}
It is easy to see that for every $n>1$
\begin{equation*}
F_n^{-1} F_n \subseteq \{(a,b,c): -n <a,b < n, - 3 n^2 < c< 3 n^2 \}.
\end{equation*}
Since $\cntm{F_n} = n^4$ for every $n$, the sequence $(F_n)_{n\geq 1}$ is tempered.
\end{exa}

For the purposes of this work we need to introduce special F{\o}lner monotilings
where one can `average' along the intersections $F_n \cap \calZ_k$ for
every fixed $k$ and $n \to \infty$. This, together with some other
requirements, leads to the following definition. We call a left F{\o}lner monotoling $([F_n,\calZ_n])_{n \geq 1}$ \textbf{regular} if the following assumptions hold:
\begin{aufzi}
\item the sequence $(F_n)_{n \geq 1}$ is a tempered two-sided F{\o}lner sequence;
\item for every $k$ the function $\indi{\calZ_k} \in \Ell{\infty}(\Gamma)$ is a good weight for pointwise convergence of
  ergodic averages along the sequence $(F_n)_{n \geq 1}$;
\item $\frac{\cntm{F_n}}{\log n} \to \infty$ as $n \to \infty$;
\item $\ue \in F_n$ for every $n$.
\end{aufzi}

Of course, our motivating example for the notion of a regular
F{\o}lner monotiling is the Example \ref{ex.zdmonotiling}. Below we explain why the corresponding indicator functions $\indi{\calZ_k}$ are good
weights for every $k$. Checking the remaining conditions for the
regularity of the F{\o}lner monotiling $([F_n,\calZ_n])_{n \geq 1}$ is
straightforward. 
\begin{exa}
\label{ex.regmonot}
Let $\Gamma$ be an amenable group with a fixed tempered F{\o}lner
sequence $(F_n)_{n \geq 1}$, $H \leq \Gamma$ be a finite index
subgroup. Let $F \subseteq \Gamma$ be the fundamental domain for left
cosets of $H$. Then $[F,H]$ is a left monotiling of
$\Gamma$. Furthermore, the indicator function $\indi{H}$ is a good
weight. To see this, consider the ergodic system
$\bfX:=(\fact{\Gamma}{H},|\cdot|_{\cdot},\Gamma)$, where $\Gamma$ acts
on the left on the finite set $\fact{\Gamma}{H}$ with normalized
counting measure $|\cdot|_{\cdot}$ by
\begin{equation*}
g (f H):=g f H, \ f \in F, g \in \Gamma.
\end{equation*}
Let $f:=\indi{\ue H} \in \Ell{\infty}{(\fact{\Gamma}{H})}$ and $x:=\ue
H \in \fact{\Gamma}{H}$. Then $\indi{H}(g) = f(g x)$ for all $g \in
\Gamma$ and the statement follows from Theorem
\ref{thm.retthm}\footnote{One can also prove this directly without
  referring to Theorem \ref{thm.retthm}.}.
\end{exa}

In what follows we will need the following simple
\begin{prop}
\label{prop.fmonot}
Let  $([F_n, \calZ_n])_{n \geq 1}$ be a left F{\o}lner monotiling of $\Gamma$ s.t. $\ue \in F_n$ for every $n$. Then for every fixed $k$
\begin{equation}
\frac{\cntm{\kintl{F_k}{F_n} \cap \calZ_k}}{\cntm{F_n}} \to \frac 1 {\cntm{F_k}}
\end{equation}
and
\begin{equation}
\frac{\cntm{F_n \cap \calZ_k}}{\cntm{F_n}} \to \frac 1 {\cntm{F_k}}
\end{equation}
as $n \to \infty$. If, additionally,  $(F_n)_{n \geq 1}$ is a
two-sided F{\o}lner sequence, then for every fixed $k$
\begin{equation}
\frac{\cntm{\kintl{F_k}{F_n} \cap \kintr{F_k^{-1}}{F_n} \cap \calZ_k}}{\cntm{F_n}} \to \frac 1 {\cntm{F_k}}
\end{equation}
as $n \to \infty$.
\end{prop}
\begin{proof}
Observe first that, under initial assumptions of the theorem, for every set $A
\subseteq \Gamma$, $k \geq 1$ and $g \in \Gamma$ we have
\begin{equation*}
g \in \kintl{F_k}{A} \Leftrightarrow F_k g \subseteq A
\end{equation*}
and
\begin{equation*}
g \in \kintr{F_k^{-1}}{A} \Leftrightarrow g F_k^{-1} \subseteq A.
\end{equation*}
Let $k \geq 1$ be fixed. For every $n \geq 1$, consider the finite set $A_{n,k}:=\{ g \in \calZ_k: F_k g \cap \kintl{F_k}{F_n} \neq \varnothing \}$. Then the translates $\{ F_k z: z \in A_{n,k} \}$ form a disjoint cover of the set $\kintl{F_k}{F_n}$.  It is easy to see that
\begin{equation*}
\Gamma = \kintl{F_k}{F_n} \sqcup \kbdrl{F_k}{F_n} \sqcup \kintl{F_k}{\comp{F_n}}. 
\end{equation*}
Since $A_{n,k} \cap \kintl{F_k}{\comp{F_n}} = \varnothing$, we can decompose the set of centers $A_{n,k}$ as follows:
\begin{equation*}
A_{n,k} = (A_{n,k} \cap \kintl{F_k}{F_n}) \sqcup (A_{n,k} \cap \kbdrl{F_k}{F_n}).
\end{equation*} 
Since $(F_n)_{n \geq 1}$ is a F{\o}lner sequence,
\begin{equation*}
\frac{\cntm{F_k (A_{n,k} \cap \kbdrl{F_k}{F_n})}}{\cntm{F_n}}=\frac{\cntm{F_k} \cdot \cntm{A_{n,k} \cap \kbdrl{F_k}{F_n}}}{\cntm{F_n}} \to 0
\end{equation*}
and $\cntm{\kintl{F_k}{F_n}}/\cntm{F_n} \to 1$ as $n \to \infty$. Then from the inequalities
\begin{align*}
\frac{\cntm{\kintl{F_k}{F_n}}}{\cntm{F_n}} &\leq \frac{\cntm{F_k (A_{n,k} \cap \kbdrl{F_k}{F_n})}}{\cntm{F_n}} + \frac{\cntm{ F_k (A_{n,k} \cap \kintl{F_k}{F_n}) }}{\cntm{F_n}} \\
&\leq \frac{\cntm{F_k (A_{n,k} \cap \kbdrl{F_k}{F_n}) }}{\cntm{F_n}} + 1
\end{align*}
we deduce that
\begin{equation}
\frac{\cntm{F_k} \cdot \cntm{A_{n,k} \cap \kintl{F_k}{F_n}}}{\cntm{F_n}} \to 1
\end{equation}
as $n \to \infty$. It remains to note that $A_{n,k} \cap
\kintl{F_k}{F_n} = \calZ_k \cap \kintl{F_k}{F_n}$ and the first
statement follows. The second statement follows trivially from the first one. To
obtain the last statement, observe that $\cntm{\kintr{F_k^{-1}}{F_n}}
/ \cntm{F_n} \to 1$ as $n \to \infty$ since $(F_n)_{n \geq 1}$ is a right F{\o}lner
sequence, thus
\begin{equation*}
\lim\limits_{n \to \infty} \frac{\cntm{\kintl{F_k}{F_n} \cap F_n \cap
    \calZ_k}}{\cntm{F_n}} = \lim\limits_{n
  \to \infty} \frac{\cntm{\kintl{F_k}{F_n} \cap \kintr{F_k^{-1}}{F_n} \cap
    \calZ_k}}{\cntm{F_n}} = \frac 1 {\cntm{F_k}}.
\end{equation*}
\end{proof}

This proposition has an important corollary, namely
\begin{thm}
\label{thm.wghtd}
Let $([F_n,\calZ_n])_{n \geq 1}$ be a regular F{\o}lner monotiling. Then for every measure-preserving system $\bfX=(\prX,\mu,\Gamma)$, every $f \in \Ell{\infty}{(\prX)}$ and every $k \geq 1$ the limits
\begin{align*}
& \cntm{F_k} \lim\limits_{n \to \infty} \avg{g \in F_n} \indi{\calZ_k} f(g \omega)  =  \lim\limits_{n \to \infty} \avg{g \in F_n \cap \calZ_k} f(g \omega) = \\
&=\lim\limits_{n \to \infty} \avg{g \in \kintl{F_k}{F_n} \cap \kintr{F_k^{-1}}{F_n} \cap \calZ_k} f(g \omega)
\end{align*}
exist and coincide for $\mu$-a.e. $\omega \in \prX$.
\end{thm}
\begin{proof}
Existence of the limit on the left hand side follows from the definition of a good weight and the definition of a regular F{\o}lner monotiling, equality of the limits follows from the previous proposition.
\end{proof}

Later in the Section \ref{ss.compfoln} we will add a \emph{computability} requirement to the notion of a regular F{\o}lner monotiling. The central result of this paper says that the Brudno's theorem holds for groups admitting a computable regular symmetric F{\o}lner monotiling. 

\subsection{Computability and Kolmogorov complexity}
\label{ss.compcompl}
In this section we will discuss the standard notions of computability
and Kolmogorov complexity that will be used in this work. We refer to Chapter 7 in \cite{hedman2004} for details, more definitions and proofs.

For a natural number $k$ a $k$-ary \textbf{partial function} is any
function of the form $f: D \to \N \cup \{ 0 \}$, where $D$,
\textbf{domain of definition}, is a subset of $(\N \cup \{ 0 \})^k$
for some natural $k$. A $k$-ary partial function is called
\textbf{computable} if there exists an algorithm which takes a
$k$-tuple of nonnegative integers $(a_1,a_2,\dots,a_k)$, prints
$f((a_1,a_2,\dots,a_k))$ and terminates if $(a_1,a_2,\dots,a_k)$ is in
the domain of $f$, while yielding no output otherwise (in
particular, it might fail to terminate). A function is called \textbf{total} if it is defined everywhere.

The term \emph{algorithm} above stands, informally speaking, for a computer program. One way to formalize it is through introducing the class of \emph{recursive functions}, and the resulting notion coincides with the class of functions computable on \emph{Turing machines}. We do not focus on these question in this work, and we will think about computability in an `informal' way.

A set $A \subseteq \N$ is called \textbf{recursive} (or \textbf{computable}) if the indicator function $\indi{A}$ of $A$ is computable. It is easy to see that finite and co-finite subsets of $\N$ are computable. Furthermore, for computable sets $A,B \subseteq \N$ their union and intersection are also computable. If a total function $f: \N \to \N$ is computable and $A \subseteq \N$ is a computable set, then $f^{-1}(A)$, the full preimage of $A$, is computable. The image of a computable set via a total computable bijection is computable, and the inverse of such a bijection is a computable function.

A sequence of subsets $(F_n)_{n \geq 1}$ of $\N$ is called \textbf{computable} if the total function $\indi{F_{\cdot}}: (n,x) \mapsto \indi{F_n}(x)$ is computable. It is easy to see that a total function $f: \N \to \N$ is computable if and only if the sequence of singletons $(f(n))_{n \geq 1}$ is computable in the sense above.

It is very often important to have a numeration of elements of a set by natural numbers. A set $A \subseteq \N$ is called \textbf{enumerable} if there exist a total computable surjective function $f: \N \to A$. If the set $A$ is infinite, we can also require $f$ to be injective. This leads to an equivalent definition because an algorithm computing the function $f$  can be modified so that no repetitions occur in its output. Finite and cofinite sets are enumerable. It can be shown (Proposition 7.44 in \cite{hedman2004}) that a set $A$ is computable if and only if both $A$ and $\N \setminus A$ are enumerable. Furthermore, for a set $A \subsetneq \N$ the following are equivalent:
\begin{aufzii}
\item $A$ is enumerable;
\item $A$ is the domain of definition of a partial recursive function.
\end{aufzii}

Finally, we can introduce the Kolmogorov complexity for finite
words. Let $A$ be a computable partial function defined on a domain
$D$ of finite binary words with values in the set of all finite words
over a finite alphabet $\Lambda$. Of course, we have defined computable functions on subsets of $(\N \cup \{ 0 \})^k$ with values in $\N \cup \{ 0 \}$ above, but this can be easily extended to (co)domains of finite words over finite alphabets. We can think of $A$ as a `decompressor' that takes compressed binary descriptions (or `programs') in its domain, and decompresses them to finite words over alphabet $\Lambda$. Then we define the \textbf{Kolmogorov complexity} of a finite word $\omega$ with respect to $A$ as follows:
\begin{equation*}
\kcc{A}{\omega}:=\inf\{ l(p): A(p)=w \},
\end{equation*}
where $l(p)$ denotes the length of the description. If some word $\omega_0$ does not admit a compressed version, then we let $\kcc{A}{\omega_0} = \infty$. The \textbf{average Kolmogorov complexity} with respect to $A$ is defined by
\begin{equation*}
\kcca{A}{\omega}:=\frac{\kcc{A}{\omega}}{l(\omega)},
\end{equation*}
where $l(\omega)$ is the length of the word $\omega$. Intuitively speaking, this quantity tells how effective the compressor $A$ is when describing the word $\omega$.

Of course, some decompressors are intuitively better than some others. This is formalized by saying that $A_1$ is \textbf{not worse} than $A_2$ if there is a constant $c$ s.t. for all words $\omega$
\begin{equation}
\label{eq.optdecomp}
\kcc{A_1}{\omega} \leq \kcc{A_2}{\omega}+c.
\end{equation}
A theorem of Kolmogorov says that there exist a decompressor $\aast$
that is optimal, i.e. for every decompressor $A$ there is a constant
$c$ s.t. for all words $\omega$ we have
\begin{equation*}
\kcc{\aast}{\omega} \leq \kcc{A}{\omega}+c.
\end{equation*}
An optimal decompressor is not unique, so from now on we let $\aast$
be a fixed optimal decompressor.

The notion of Kolmogorov complexity can be extended to words defined on finite subsets of $\N$, and this will be essential in the following sections. More precisely, let $X \subseteq \N$ be a finite subset, $\imath_X: X \to \{ 1,2,\dots, \card X\} $ an increasing bijection, $\Lambda$ a finite alphabet, $A$ a decompressor and $\omega \in \Lambda^Y$ a word defined on some set $Y \supseteq X$. Then we let
\begin{equation}
\label{eq.kcnsubs}
\uK_A(\omega,X):=\kcc{A}{\omega \circ \imath_X^{-1}}.
\end{equation}
and
\begin{equation}
\label{eq.kcansubs}
\overline{\uK}_A(\omega,X):=\frac{\kcc{A}{\omega \circ \imath_X^{-1}}}{\card X}.
\end{equation}
We call $\uK_A(\omega,X)$ the \textbf{Kolmogorov complexity} of $\omega$ over $X$ with respect to $A$, and $\overline{\uK}_A(\omega,X)$ is called the \textbf{mean Kolmogorov complexity} of $\omega$ over $X$ with respect to $A$. If a decompressor $A_1$ is not worse than a decompressor $A_2$ with some constant $c$, then for all $X, \omega$ above
\begin{equation*}
\uK_{A_1}(\omega,X)\leq\uK_{A_2}(\omega,X)+c.
\end{equation*}

If $X \subseteq \N$ is an infinite subset and $(F_n)_{n \geq 1}$ is a sequence of finite subsets of $X$ s.t. $\card{F_n} \to \infty$, then the asymptotic Kolmogorov complexity of $\omega \in \Lambda^X$ with respect to $(F_n)_{n \geq 1}$ and a decompressor $A$ is defined by
\begin{equation*}
\widehat{\uK}_A(\omega):=\limsup\limits_{n \to \infty}{\overline{\uK}_A(\omega|_{F_n},F_n)}.
\end{equation*}
The dependence on the sequence $(F_n)_{n \geq 1}$ is omitted in the notation. It is easy to see that for every decompressor $A$ and $\omega \in \Lambda^X$
\begin{equation}
\label{eq.optdeclim1}
\widehat{\uK}_{\aast}(\omega) \leq \widehat{\uK}_A(\omega).
\end{equation}

From now on, we will (mostly) use the optimal decompressor $\aast$ and
write $\uK(\omega,X)$, $\overline{\uK}(\omega,X)$ and
$\widehat{\uK}(\omega)$ omitting an explicit reference to $\aast$.

When estimating the Kolmogorov complexity of words we will often have to encode nonnegative integers using binary words. We will now fix some notation that will be used later. When $n$ is a nonnegative integer, we write
$\underline \un$ for the \textbf{binary encoding} of $n$ and $\overline \un$ for the \textbf{doubling encoding} of $n$, i.e. if $b_l b_{l-1} \dots b_0$ is the binary expansion of $n$, then $\underline \un$ is the binary word $\ub_l \ub_{l-1} \dots \ub_0$ of length $l+1$ and $\overline \un$ is the binary word $\ub_l \ub_l \ub_{l-1} \ub_{l-1} \dots \ub_0 \ub_0$ of length $2l+2$. We denote the length of the binary word $\uw$ by $l(\uw)$, and is clear that $l(\underline \un) \leq \lfloor \log n\rfloor+1$ and $l(\overline \un) \leq 2 \lfloor \log n\rfloor+2$. We write $\widehat \un$ for the encoding $\overline{l(\underline \un)} \del \underline \un$ of $n$, i.e. the encoding begins with the length of the binary word $\underline \un$ encoded using doubling encoding, then the delimiter $\del$ follows, then the word $\underline n$. It is clear that $l(\widehat \un) \leq  2 \lfloor \log(\lfloor \log n \rfloor + 1) \rfloor + \lfloor \log n \rfloor + 5$. This encoding enjoys the following property: given a binary string
\begin{equation*}
\widehat x_1 \widehat x_2 \dots \widehat x_l,
\end{equation*}
the integers $x_1,\dots,x_l$ are unambiguously restored. We will call such an encoding a \textbf{simple prefix-free encoding}.

\subsection{Computable spaces, word presheaves and complexity}
\label{ss.compspaces}
The goal of this section is to introduce the notions of \emph{computable space}, \emph{computable function} between computable spaces and \emph{word presheaf} over computable spaces. The complexity of sections of word presheaves and asymptotic complexity of sections of word presheaves are introduced in this section as well.

An \textbf{indexing} of a set $X$ is an injective mapping $\imath: X \to \N$ such that $\imath(X)$ is a computable subset. Given an element $x \in X$, we call $\imath(x)$ the \textbf{index} of $x$. If $i \in \imath(X)$, we denote by $x_i$ the element of $X$ having index $i$. A \textbf{computable space} is a pair $(X, \imath)$ of a set $X$ and an indexing $\imath$. Preimages of computable subsets of $\N$ under $\imath$ are called \textbf{computable subsets} of $(X,\imath)$. Each computable subset $Y \subseteq X$ can be seen as a computable space $(Y, \imath|_Y)$, where $\imath|_Y$ is the restriction of the indexing function. Of course, the set $\N$ with identity as an indexing function is a computable space, and the computable subsets of $(\N, \id)$ are precisely the computable sets of $\N$ in the sense of Section \ref{ss.compcompl}.

Let $(X_1, \imath_1), (X_2, \imath_2), \dots, (X_k, \imath_k),(Y,\imath)$ be computable spaces. A (total) function $f: X_1\times X_2 \times \dots \times X_k \to Y$ is called \textbf{computable} if the function $\widetilde f: \imath_1(X_1) \times \imath_2(X_2) \times \dots \times \imath_k(X_k) \to \imath(Y)$ determined by the condition
\begin{equation*}
\widetilde f(\imath_1(x_1),\imath_2(x_2),\dots,\imath_k(x_k)) = \imath(f(x_1,x_2,\dots,x_k))
\end{equation*}
for all $(x_1,x_2,\dots,x_k) \in X_1 \times X_2 \times \dots \times X_k$ is computable. This definition extends the standard definition of computability from Section \ref{ss.compcompl} when the computable spaces under consideration are $(\N,\id)$. A computable function $f: (X,\imath_1) \to (Y,\imath_2)$ is called a \textbf{morphism} between computable spaces. This yields the definition of the \textbf{category of computable spaces}. Let $(X, \imath_1)$, $(X, \imath_2)$ be computable spaces. The indexing functions $\imath_1$ and $\imath_2$ of $X$ are called \textbf{equivalent} if $\id: (X,\imath_1) \to (X,\imath_2)$ is an isomorphism. It is clear that the classes of computable functions and computable sets do not change if we pass to equivalent indexing functions.

Given a computable space $(X,\imath)$, we call a sequence of subsets  $(F_n)_{n \geq 1}$ of $X$ \textbf{computable} if the function $\indi{F_{\cdot}}: \N \times X \to \{ 0,1\}, (n,x) \mapsto \indi{F_n}(x)$ is computable. We will also need a special notion of computability for sequences of \emph{finite} subsets of $(X,\imath)$. A sequence of finite subsets $(F_n)_{n \geq 1}$ of $X$ is called \textbf{canonically computable} if there is an algorithm that, given $n$, prints the set $\imath(F_n)$ and halts. One way to make this more precise is by introducing the canonical index of a finite set. Given a finite set $A=\{ x_1,x_2,\dots,x_k\} \subset \N$, we call the number $\cin{A}:=\sum\limits_{i=1}^k 2^{x_i}$ the \textbf{canonical index} of $A$. Hence a sequence of finite subsets $(F_n)_{n \geq 1}$ of $X$ is canonically computable if and only if the total function $n \mapsto \cin{\imath(F_n)}$ is computable. Of course, a canonically computable sequence of finite sets is computable, but the converse is not true due to the fact that there is no effective way of determining how large a finite set with a given computable indicator function is. It is easy to see that the class of canonically computable sequences of finite sets does not change if we pass to an equivalent indexing. The proof of the following proposition is straightforward:

\begin{prop}
Let $(X,\imath)$ be a computable space. Then
\begin{aufzi}
\item If $(F_n)_{n \geq 1}, (G_n)_{n \geq 1}$ are computable
  (resp. canonically computable) sequences of sets, then the sequences
  of sets $(F_n \cup G_n)_{n \geq 1}$, $(F_n \cap G_n)_{n \geq 1}$ and
  $(F_n \setminus G_n)_{n \geq 1}$ are computable (resp. canonically computable).
\item If $(F_n)_{n \geq 1}$ is a canonically computable sequence of sets and $(G_n)_{n \geq 1}$ is a computable sequence of sets, then the sequence of sets $(F_n \cap G_n)_{n \geq 1}$ is canonically computable.
\end{aufzi}

\end{prop}

Let $(X,\imath)$ be a computable space and $\Lambda$ be a finite alphabet. A \textbf{word presheaf} $\calF_{\Lambda}$ on $X$ consists of
\begin{aufziii}
\item A set $\calF_{\Lambda}(U)$ of $\Lambda$-valued functions defined on the set $U$ for every computable subset $U \subseteq X$;
\item A restriction mapping $\rho_{U,V}: \calF_{\Lambda}(U) \to \calF_{\Lambda}(V)$ for each pair $U,V$ of computable subsets s.t. $V \subseteq U$, that takes functions in $\calF_{\Lambda}(U)$ and restricts them to the subset $V$.
\end{aufziii}
It is easy to see that the standard `presheaf axioms' are satisfied: $\rho_{U,U}$ is identity on $\calF_{\Lambda}(U)$ for every computable $U \subseteq X$, and for every triple $V \subseteq U \subseteq W$ we have that $\rho_{W,V} = \rho_{U,V} \circ \rho_{W,U}$. Elements of $\calF_{\Lambda}(U)$ are called \textbf{sections} over $U$, or \textbf{words} over $U$. We will often write $s|_V$ for $\rho_{U,V}s$, where $s \in \calF_{\Lambda}(U)$ is a section.

We have introduced Kolmogorov complexity of words supported on subsets of $\N$ in the previous section, now we want to extend this by introducing complexity of sections. Let $(X,\imath)$ be a computable space and let $\calF_{\Lambda}$ be a word presheaf over $(X,\imath)$. Let $U \subseteq X$ be a finite set and $\omega \in \calF_{\Lambda}(U)$. Then we define the \textbf{Kolmogorov complexity} of $\omega \in \calF_{\Lambda}(U)$ by
\begin{equation}
\kc{\omega}{U}:=\kc{\omega \circ \imath^{-1}}{\imath(U)}
\end{equation}
and the \textbf{mean Kolmogorov complexity} of $\omega \in \calF_{\Lambda}(U)$ by
\begin{equation}
\kca{\omega}{U}:=\kca{\omega \circ \imath^{-1}}{\imath(U)}.
\end{equation}
The quantities on the right hand side here are defined in the Equations \ref{eq.kcnsubs} and \ref{eq.kcansubs} respectively (which are special cases of the more general definition when the computable space $X$ is $(\N,\id)$).

Let $(F_n)_{n \geq 1}$ be a sequence of finite subsets of $X$ s.t. $\card F_n \to \infty$. Then we define \textbf{asymptotic Kolmogorov complexity} of a section $\omega \in \calF_{\Lambda}(X)$ along the sequence $(F_n)_{n \geq 1}$ by
\begin{equation*}
\kcl{\omega}:= \limsup\limits_{n \to \infty} \kca{\omega|_{F_n}}{F_n}.
\end{equation*}
Dependence on the sequence $(F_n)_{n \geq 1}$ is omitted in the notation for $\widehat{\uK}$, but it will be always clear from the context which sequence we take.

We close this section with an interesting result on invariance of asymptotic Kolmogorov complexity. It says that asymptotic Kolmogorov complexity of a section $\omega \in \calF_{\Lambda}(X)$ does not change if we pass to an equivalent indexing.
\begin{thm}[Invariance of asymptotic complexity]
\label{thm.asinv}
Let $\imath_1, \imath_2$ be equivalent indexing functions of a set $X$. Let $(F_n)_{n \geq 1}$ be a sequence of finite subsets of $X$ such that
\begin{aufzi}
\item $(F_n)_{n \geq 1}$ is a canonically computable sequence of sets in $(X,\imath_1)$;
\item $\frac{\card F_n}{\log n} \to \infty$ as $n \to \infty$.
\end{aufzi}
Let $\omega \in \calF_{\Lambda}(X)$. Then
\begin{equation*}
\limsup\limits_{n \to \infty} \kca{\omega|_{F_n} \circ \imath_1^{-1}}{\imath_1(F_n)} = \limsup\limits_{n \to \infty} \kca{\omega|_{F_n} \circ \imath_2^{-1}}{\imath_2(F_n)},
\end{equation*}
i.e. asymptotic Kolmogorov complexity of $\omega$ does not change when we pass to an equivalent indexing.
\end{thm}
\begin{proof}
Since the indexing functions $\imath_1, \imath_2$ are equivalent, there is a computable bijection $\phi: \imath_2(X) \to \imath_1(X)$ such that $\phi(\imath_2(x)) = \imath_1(x)$ for all $x \in X$. Furthermore, the sequence $(F_n)_{n \geq 1}$ is canonically computable in $(X,\imath_2)$.

Let $n$ be fixed. By the definition,
\begin{align*}
\kca{\omega|_{F_n} \circ \imath_1^{-1}}{\imath_1(F_n)} = \frac{\kcc{\aast}{(\omega|_{F_n} \circ \imath_1^{-1}) \circ \imath_{\imath_1(F_n)}^{-1}}}{\card F_n},
\end{align*}
where $\omega|_{F_n} \circ \imath_1^{-1}$ is seen as a word on $\imath_1(F_n) \subseteq \N$ and $\widetilde \omega_1:=(\omega|_{F_n} \circ \imath_1^{-1}) \circ \imath_{\imath_1(F_n)}^{-1}$ is a word on $\{ 1,2,\dots,\card F_n \} \subseteq \N$. Let $\up_1$ be an optimal description of $(\omega|_{F_n} \circ \imath_1^{-1}) \circ \imath_{\imath_1(F_n)}^{-1}$. Similarly, $\widetilde \omega_2:=(\omega|_{F_n} \circ \imath_2^{-1}) \circ \imath_{\imath_2(F_n)}^{-1}$ is a word on $\{ 1,2,\dots,\card F_n \}$. It is clear that $\widetilde \omega_1$ is a permutation of $\widetilde \omega_2$, hence we can describe $\widetilde \omega_2$ by giving the description of $\widetilde \omega_1$ and saying how to permute it to obtain $\widetilde \omega_2$. We make this intuition formal below.

We define a new decompressor $A'$. The domain of definition of $A'$ consists of the programs of the form
\begin{equation}
\overline \ul \del \up,
\end{equation}
where $\overline \ul$ is a doubling encoding of an integer $l$ and $\up$ is an input for $\aast$. The decompressor works as follows. Compute the subsets $\imath_1(F_l)$ and $\imath_2(F_l)$ of $\N$. We let $\overline \phi$ be the element of $\sym{\card F_n}$ such that the diagram
\begin{displaymath}
    \xymatrix{ \imath_1(F_n) \ar[d]_{\imath_{\imath_1(F_n)}} & & \imath_2(F_n) \ar[d]^{\imath_{\imath_2(F_n)}} \ar[ll]^{\phi} \\
               \{ 1,2,\dots,\card F_n \}   & & \{ 1,2,\dots,\card F_n \} \ar[ll]^{\overline \phi} }
\end{displaymath}
commutes. We compute the word $\omega':=\aast(\up)$, and if $\card F_l
\neq l(\omega')$ the algorithm terminates without producing
output. Otherwise, the word $\omega' \circ \overline \phi$ is
printed. It follows that there is a constant $c$ such that the
following holds: for all $l \in \N$ and for all words $\omega'$ of
length $\card F_l$ we have
\begin{equation*}
\kcc{\aast}{\omega' \circ \overline \phi} \leq \kcc{\aast}{\omega'}+2 \log l+c,
\end{equation*}
where $\overline \phi$ is the permutation of $\{ 1,2,\dots,\card F_l\}$
defined above.

Finally, consider the program $\up':=\overline \un \del \up_1$, then
$A'(\up') = \widetilde \omega_2$. We deduce that
$\kcc{\aast}{\widetilde \omega_2} \leq \kcc{\aast}{\widetilde
  \omega_1} + 2 \log n + c$. The statement of the theorem follows trivially.
\end{proof}

To simplify the notation in the following sections, we adopt the
following convention. We say explicitly what indexing function we use
when introducing a computable space, but later, when the indexing is
fixed, we often omit the indexing function from the notation and think about computable spaces as computable subsets of $\N$. Words defined on subsets of a computable space become words defined on subsets of $\N$. This will help to simplify the notation without introducing much ambiguity.

\subsection{Computable Groups}
\label{ss.compgroup}
In this section we provide the definitions of a computable group and a few related notions, connecting results from algebra with computability. This section is based on \cite{rabin1960}.

Let $\Gamma$ be a group with respect to the multiplication operation
$\ast$. An indexing $\imath$ of $\Gamma$ is called \textbf{admissible}
if the function $\ast: (\Gamma,\imath) \times (\Gamma,\imath) \to
(\Gamma,\imath)$ is a computable function in the sense of Section
\ref{ss.compspaces}. A \textbf{computable group} is a pair
$(\Gamma,\imath)$ of a group $\Gamma$ and an admissible indexing
$\imath$. 

Of course, the groups $\Z^d$ and $\UT{d}{\Z}$ possess
`natural' admissible indexings. More precisely, for the group $\Z$ we fix the indexing
\begin{equation*}
\imath: n \mapsto 2|n| + \indi{n \geq 0},
\end{equation*}
which is admissible. Next, it is clear that for every $d > 1$ the group $\Z^d$ possesses an admissible
indexing function such that all coordinate projections onto $\Z$,
endowed with the indexing function $\imath$ above, are
computable. Similarly, for every $d \geq 2$ the group $\UT{d}{\Z}$
possesses an admissible indexing function such that for every pair of
indices $1 \leq i,j \leq d$
the evaluation function sending a matrix $g \in \UT{d}{\Z}$ to its
$(i,j)$-th entry is a computable function to $\Z$. We leave the details to the
reader. It does not matter which admissible indexing function of $\Z^d$ or
$\UT{d}{\Z}$ we use as long as it satisfies the conditions above, so
from now on we assume that this choice is fixed.

The following lemma from \cite{rabin1960} shows that in a computable group taking the inverse is also a computable operation.
\begin{lemma}
Let $(\Gamma, \imath)$ be a computable group. Then the function $\inv:
(\Gamma,\imath) \to (\Gamma,\imath)$, $ g \mapsto g^{-1}$ is computable.
\end{lemma}

$(\Gamma, \imath)$ is a computable space, and we can talk about computable subsets of $(\Gamma, \imath)$. A subgroup of $\Gamma$ which is a computable subset will be called a \textbf{computable subgroup}. A homomorphism between computable groups that is computable as a map between computable spaces will be called a \textbf{computable homomorphism}. The proof of the proposition below is straightforward.
\begin{prop}
Let $(\Gamma,\imath)$ be a computable group. Then the following assertions holds
\begin{aufziii}
\item Given a computable set $A \subseteq \Gamma$ and a group element $g \in \Gamma$, the sets $A^{-1}, gA$ and $Ag$ are computable;

\item Given a computable (resp. canonically computable) sequence
  $(F_n)_{n \geq 1}$ of subsets  of $\Gamma$ and a group element $g \in \Gamma$, the
  sequences $(g F_n)_{n \geq 1}, (F_n g)_{n \geq 1}$ are computable
  (resp. canonically computable).
\end{aufziii}
\end{prop}

It is interesting to see that a computable version of the `First Isomorphism Theorem' also holds.
\begin{thm}
Let $(G,\imath)$ be a computable group and let $(H,\imath|_H)$ be a computable normal subgroup, where $\imath|_H$ is the restriction of the indexing function $\imath$ to $H$. Then there is a compatible indexing function $\imath'$ on the factor group $\fact G H$ such that the quotient map $\pi: (G,\imath) \to (\fact G H, \imath')$ is a computable homomorphism.
\end{thm}
For the proof we refer the reader to the Theorem 1 in \cite{rabin1960}.

\subsection{Computable F{\o}lner sequences and computable monotilings}
\label{ss.compfoln}
The notions of an amenable group and a F{\o}lner sequence are well-known, but, since we are working with computable groups, we need to develop their `computable' versions.

Let $(\Gamma,\imath)$ be a computable group. A left F{\o}lner monotiling  $([F_n, \calZ_n])_{n \geq 1}$ of $\Gamma$ is called \textbf{computable} if the following assertions hold
\begin{aufzi}
\item $(F_n)_{n \geq 1}$ is a canonically computable sequence of finite subsets of $\Gamma$;
\item $(\calZ_n)_{n \geq 1}$ is a computable sequence of subsets of $\Gamma$.
\end{aufzi}

First of all, let us show that the regular symmetric monotiling $([F_n,\calZ_n])_{n \geq
1}$ of $\Z^d$ from Example \ref{ex.zdmonotiling} is computable.
\begin{exa}
\label{ex.zdmonot}
Consider the group $\Z^d$ for some $d \geq 1$. We remind the reader
that it is endowed with an
admissible indexing such that all the coordinate projections $\Z^d \to
\Z$ are computable. Then the F{\o}lner sequence $F_n=
[0,1,2,\dots,n-1]^d$ is canonically computable. Furthermore, the
corresponding sets of centers equal $n \Z^d$ for every $n$, hence $([\calZ_n,F_n])_{n \geq 1}$ is a computable regular symmetric F{\o}lner monotiling.
\end{exa}

Next, we return to Example \ref{ex.hmonotiling}.
\begin{exa}
\label{ex.hcompmonot}
Consider the group $\UT{3}{\Z}$ and the monotiling $([F_n,\calZ_n])_{n
\geq 1}$ from Example \ref{ex.hmonotiling} given by
\begin{equation*}
\calZ_n=\{ (a,b,c) \in \UT{3}{\Z}: a,b \in n \Z, c \in n^2 \Z \}
\end{equation*}
and 
\begin{equation*}
F_n=\{ (a,b,c) \in \UT{3}{\Z}: 0 \leq a,b < n, 0 \leq c < n^2 \}
\end{equation*}
for every $n \geq 1$. We define the projections $\pi_1,\pi_2,\pi_3:
\UT{3}{\Z} \to \Z$ as follows. For every $g = (a,b,c) \in \UT{3}{\Z}$ we
let
\begin{align*}
&\pi_1(g) := a, \\
&\pi_2(g) := b, \\
&\pi_3(g) := c.
\end{align*}
The functions $\pi_1,\pi_2,\pi_3$ are computable. By definition, for every
$(n,g) \in \N \times \UT{3}{\Z}$ 
\begin{equation*}
\indi{\calZ_{\cdot}}(n,g) = 1 \Leftrightarrow (\pi_1(g) \in n \Z) \wedge
(\pi_2(g) \in n \Z) \wedge (\pi_3(g) \in n^2 \Z), 
\end{equation*}
hence the sequence of sets $(\calZ_n)_{n \geq 1}$ is computable. It is
also trivial to show that the sequence $(F_n)_{n \geq 1}$ is
canonically computable.

It follows that $([F_n,\calZ_n])_{n \geq 1}$ is a computable regular
symmetric F{\o}lner monotiling.
\end{exa}

In general, checking temperedness of a given canonically computable
F{\o}lner sequence is not trivial. Lindenstrauss in
\cite{lindenstrauss2001} proved that every F{\o}lner sequence has a
tempered F{\o}lner subsequence. Furthermore, the construction of a
tempered F{\o}lner subsequence from a given F{\o}lner sequence is
`algorithmic'. We provide his proof below, and we will use this result
later in this section when discussing F{\o}lner monotilings of
$\UT{d}{\Z}$ for $d > 3$.
\begin{prop}
\label{pr.tempsseq}
Let $(F_n)_{n \geq 1}$ be a canonically computable F{\o}lner sequence in a computable group $(\Gamma,\imath)$. Then there is a computable function $i \mapsto n_i$ s.t. the subsequence $(F_{n_i})_{i \geq 1}$ is a canonically computable tempered F{\o}lner subsequence.
\end{prop}
\begin{proof}
We define $n_i$ inductively as follows. Let $n_1:=1$. If $n_1,\dots,n_i$ have been determined, we set $\widetilde F_i:= \bigcup\limits_{j \leq i} F_{n_j}$. Take for $n_{i+1}$ the first integer greater than $i+1$ such that
\begin{equation*}
\cntm{F_{n_{i+1}}  \sdif \widetilde F_i^{-1} F_{n_{i+1}}} \leq \frac 1 {\cntm{\widetilde F_i}}.
\end{equation*}
The function $i \mapsto n_i$ is total computable. It follows that
\begin{equation*}
\cntm{ \bigcup\limits_{j \leq i} F_{n_j}^{-1} F_{n_{i+1}}} \leq 2 \cntm{ F_{n_{i+1}}},
\end{equation*}
hence the sequence $(F_{n_i})_{i \geq 1}$ is $2$-tempered. Since the F{\o}lner sequence $(F_n)_{n \geq 1}$ is canonically computable and the function $i \mapsto n_i$ is computable, the F{\o}lner sequence $(F_{n_i})_{i \geq 1}$ is canonically computable and tempered.
\end{proof}

In case of the discrete Heisenberg group $\UT{3}{\Z}$ we were able to
give simple formulas for the sequences $(F_n)_{n \geq 1}$ and
$(\calZ_n)_{n \geq 1}$, in particular, checking the computability was
trivial. This is no longer the case when $d>3$, and we will need the
following lemma to check the computability of the sequence $(\calZ_n)_{n \geq 1}$.
\begin{prop}
\label{prop.monoteq}
Let $(\Gamma, \imath)$ be a computable group. Let $([F_n,\calZ_n])_{n \geq 1}$ be a left F{\o}lner monotiling of $\Gamma$ such that $(F_n)_{n \geq 1}$ is a canonically computable sequence of finite sets and $\ue \in F_n$ for all $n \geq 1$. Then the following assertions are equivalent:
\begin{aufzii}
\item There is a total computable function $\phi: \N^2 \to \Gamma$ such that
\begin{equation*}
\calZ_n = \{ \phi(n,1),\phi(n,2), \dots\}
\end{equation*}
for every $n \geq 1$.
\item The sequence of sets $(\calZ_n)_{n \geq 1}$ is computable.
\end{aufzii}
\end{prop}
\begin{proof}
One implication is clear. For the converse, note that to prove computability of the function $\indi{\calZ_{\cdot}}$ we have to devise an algorithm that, given $n\in \N$ and $g \in \Gamma$, decides whether $g \in \calZ_n$ or not. Let $\phi: \N^2 \to \Gamma$ be the function from assertion (i). Then the following algorithm answers the question. Start with $i:=1$ and compute $\ue \phi(n,i), h_{1,n} \phi(n,i),\dots, h_{k,n} \phi(n,i)$, where $F_n=\{ \ue, h_{1,n},\dots,h_{k,n}\}$. This is possible since $(F_n)_{n \geq 1}$ is a canonically computable sequence of finite sets. If $g = \ue \phi(n,i) $, then the answer is `Yes' and we stop the program. If $g =   h_{j,n} \phi(n,i)$ for some $j$, then the answer is `No' and we stop the program. If neither is true, then we set $i:=i+1$ and go to the beginning.

Since $\Gamma =  F_n \calZ_n$ for every $n$, the algorithm terminates for every input.
\end{proof}

In this last example we will explain, referring to the work
\cite{golodets2002} for details, why the groups $\UT{d}{\Z}$ for $d >
3$ have computable regular symmetric F{\o}lner monotilings as well.
\begin{exa}
\label{ex.utz}
Let $d$ be fixed. Let $u_{ij}$ be the matrix whose entry with the
indices $(i,j)$ is $1$, and where all the other entries are zero. Let
$T_{ij}:=I+u_{ij}$. Let $p$ be a prime number. For every $m$ consider
the subgroup $\calZ_m$ generated by $T_{ij}^{p^{m(j-i)}}$ for all
indices $(i,j)$, $i<j$. Then $\calZ_m$ is an enumerable subset.  There exists
a total computable function $\phi: \N^2
\to \UT{d}{\Z}$ such that
\begin{equation*}
\calZ_m = \{ \phi(m,1),\phi(m,2),\phi(m,3),\dots \}
\end{equation*}
for all $m \geq 1$.

$\calZ_m$ is finite a
index subgroup of $\UT{d}{\Z}$ for every $m$. The fundamental domain
$\rho_m$ for $\calZ_m$ can be written as
\begin{align*}
\rho_m:=&\{ T_{d-1,d}^{k_{d-1,d}} \cdots T_{1,d}^{k_{1,d}}: \\
&l_{d-1,d}(m) \leq k_{d-1,d} \leq L_{d-1,d}(m), \dots, l_{1,d}(m) \leq k_{1,d} \leq L_{1,d}(m)\},
\end{align*}
where
\begin{equation*}
l_{i,j}(m) = - \lfloor \frac{p^{m(j-i)}}{2} \rfloor, \ \ L_{i,j}(m) = \lfloor \frac{p^{m(j-i)} + 1}{2} \rfloor.
\end{equation*}
It is clear that the sequence of sets $m \mapsto \rho_m$ is
canonically computable. Furthermore, it is shown in
\cite{golodets2002} that $(\rho_m)_{m \geq 1}$ is a two-sided
F{\o}lner sequence. Computablity of the F{\o}lner monotiling
$([\rho_m,\calZ_m])_{m \geq 1}$ follows from Proposition \ref{prop.monoteq}.

The fact that the F{\o}lner monotiling
$([\rho_m,\calZ_m])_{m \geq 1}$ is symmetric is clear since $\calZ_m$
is a subgroup for every $m$. The fact that for each $m$ the function
$\indi{\calZ_m}$ is a good weight along a tempered subsequence of
$(\rho_m)_{m \geq 1}$ follows from Example
\ref{ex.regmonot}. It is clear that we can ensure the growth conditions by picking a subsequence $(n_i)_{i \geq 1}$ computably s.t. $([\rho_{n_i},\calZ_{n_i}])_{i\geq 1}$ is a computable regular symmetric F{\o}lner monotiling.
\end{exa}

\section{A theorem of Brudno}
\label{s.brudno}

We are now ready to prove the main theorem of this article. First, we will explain some definitions.

By a \textbf{subshift} $(\prX,\Gamma)$ we mean a closed
$\Gamma$-invariant subset $\prX$ of $\Lambda^{\Gamma}$, where
$\Lambda$ is the finite \textbf{alphabet} of $\prX$. The left action of the group $\Gamma$ on $\prX$ is given by
\begin{equation*}
(g\cdot \omega)(x) := \omega(x g) \ \ \ \forall x,g \in \Gamma, \omega \in \prX.
\end{equation*}
The words consisting of letters from the alphabet $\Lambda$ will be often called \textbf{$\Lambda$-words}. Of course, we can assume without loss of generality that $\Lambda = \{ 1,2,\dots,k\}$ for some $k$. When an invariant probability measure $\mu$ is fixed on $\prX$, we will often denote by $\bfX=(\prX,\mu,\Gamma)$ the associated measure-preserving system. 
We can associate a word presheaf $\calF_{\Lambda}$ to the subshift $(\prX,\Gamma)$ by setting
\begin{equation}
\calF_{\Lambda}(F):=\{ \omega|_F: \omega \in \prX\}.
\end{equation}
That is, $\calF_{\Lambda}(F)$ is the set of all restrictions of words in $\prX$ to the set $F$ for every computable $F$.

The main result of this article is
\begin{thm}
\label{thm.brudno}
Let $(\Gamma,\imath)$ be a computable group with a fixed computable regular symmetric F{\o}lner monotiling  $([F_n, \calZ_n])_{n \geq 1}$. Let $(\prX,\Gamma)$ be a subshift on $\Gamma$, $\mu \in \Probm_{\Gamma}(\prX)$ be an ergodic measure and $\bfX=(\prX,\mu,\Gamma)$ be the associated measure-preserving system. Then
\begin{equation*}
\kcl{\omega} = h(\bfX)
\end{equation*}
for $\mu$-a.e. $\omega \in \prX$, where the asymptotic complexity is computed with respect to the sequence $(F_n)_{n \geq 1}$.
\end{thm}

The proof is split into two parts, establishing respective inequalities in Theorems \ref{thm.brudnogeq} and \ref{thm.brudnoleq}. From now on, we follow the strategy of the original Brudno's paper \cite{brudno1982} more or less.

Given a subshift $\prX \subseteq \Lambda^{\Gamma}$ with an invariant measure $\mu$ on the alphabet $\Lambda=\{ 1,\dots,k\}$, we define the partition
\begin{equation*}
\alpha_{\Lambda}:=\{A_1,\dots,A_k\}, \ A_i:=\{ \omega \in \prX: \omega(\ue) = i\} \text{ for } i = 1,\dots,k.
\end{equation*}
Then $\alpha_{\Lambda}$ is, clearly, a generating partition. We will use the following well-known
\begin{prop}
Let $\prX \subseteq \Lambda^{\Gamma}$ be a subshift, $\mu$ be an invariant measure on $\prX$, $\bfX=(\prX,\mu,\Gamma)$ be the associated measure-preserving system and $\alpha_{\Lambda}$ be the partition defined above. Then
\begin{equation*}
h_{\mu}(\alpha_{\Lambda},\Gamma) = h(\bfX).
\end{equation*}
\end{prop}

Given a word $\omega \in \prX$ and a finite subset $F \subseteq \Gamma$, we will set $X_F(\omega):=\{ \sigma \in \prX: \sigma|_F = \omega|_F \}$, i.e. $X_F(\omega)$ is the cylinder of all words in $\prX$ that coincide with $\omega$ when restricted to $F$. Note that
\begin{equation}
X_F(\omega) = \left(\bigvee\limits_{g \in F} g^{-1} \alpha_{\Lambda}\right)(\omega) = \alpha_{\Lambda}^F(\omega),
\end{equation}
i.e. the cylinder set $X_F(\omega)$ is precisely the atom of the partition $\alpha_{\Lambda}^F$ that contains $\omega$.

The alphabet $\Lambda$ is finite, so we encode each letter of $\Lambda$ using precisely $\lfloor \log \card \Lambda \rfloor+1$ bits. Then binary words of length $N\left(\lfloor \log \card \Lambda \rfloor+1 \right)$ are unambiguously interpreted as $\Lambda$-words of length $N$.

\subsection{Part A}
\label{ss.parta}
The first step is proving that the Kolmogorov complexity of a word
over $\Gamma$ is shift-invariant. In the proof below it will become
apparent why we need computability structure on the group and why we
require the F{\o}lner sequence to be computable. In the proof below we
follow the convention suggested at the end of Section
\ref{ss.compspaces}, i.e. we view $\Gamma$ as a computable subset of $\N$ such that the multiplication is computable.

\begin{thm}[Shift invariance]
\label{thm.shiftinv}
Let $(\Gamma, \imath)$ be a computable amenable group with a fixed canonically computable right F{\o}lner sequence $(F_n)_{n \geq 1}$ such that $\frac{\cntm{F_n}}{\log n} \to \infty$ as $n \to \infty$. Let $(\prX,\Gamma)$ be a subshift and $\omega \in \prX$ be a word on $\Gamma$. Then for every $g \in \Gamma$
\begin{equation*}
\widehat \uK(\omega) = \widehat \uK(g \cdot \omega),
\end{equation*}
where the asymptotic complexity is computed with respect to the sequence $(F_n)_{n \geq 1}$.
\end{thm}
\begin{proof}
We will prove the following claim: for arbitrary $g \in \Gamma$
\begin{equation*}
\widehat \uK(g \cdot \omega) = \limsup\limits_{n \to \infty} \frac{\kcc{\aast}{(g \cdot \omega)|_{F_n} \circ \imath_{F_n}^{-1}}}{\cntm{F_n}} \leq \kcl{\omega}.
\end{equation*}
It is trivial to see that the statement of the theorem follows from this claim. Speaking informally, our idea behind the proof of the claim is that the sets $F_n$ and $F_n g^{-1}$ are almost identical for large enough $n$. The word $(g \cdot \omega)|_{F_n \cap F_n g^{-1}}$ can be encoded using the knowledge of the word $\omega|_{F_n}$ and the \emph{computable} action by $g$ that `permutes'  a part of the word $\omega|_{F_n}$. To encode the word $(g \cdot \omega)|_{F_n}$ we also need to treat the part outside the intersection. We use the fact that our F{\o}lner sequence is computable, i.e. there is an algorithm that, given $n$, will print the set $F_n$. But then we also know the remainder $F_n \setminus F_n g^{-1}$, which is endowed with the ambient numbering of $\Gamma \subseteq \N$. Hence we can simply list additionally the $\cntm{F_n \setminus F_n g^{-1}}$ corrections we need to make, which takes little space compared to $\cntm{F_n}$. Together this implies that the complexity of $(g \cdot \omega) |_{F_n}$ can be asymptotically bounded by the complexity of $\omega|_{F_n}$. Below we make this intuition formal.

Recall that $\aast$ is a fixed asymptotically optimal decompressor in the definition of Kolmogorov complexity $\uK$. We now introduce a new decompressor $\adag$ on the domain of programs of the form
\begin{equation}
\label{eq.aprinp}
\overline{\us} \del \uw \del \overline \un \del \overline \um \del \up,
\end{equation}
where $\overline \us$ is a doubling encoding of a nonnegative integer $s$, and $\uw$ is a binary encoding of a $\Lambda$-word $\upsilon$ of length $s$, hence $l(\uw) = s (\lfloor \log \card \Lambda \rfloor + 1)$. Next, $\overline \un$ and $\overline \um$ are doubling encodings of some natural numbers $n,m$. The remainder $\up$ is required to be a valid input for $\aast$. Observe that programs of this form (Equation \ref{eq.aprinp}) are unambiguously interpreted.

Decompressor $\adag$ is defined as follows. Let $g:=g_{m}$ be the
element of the computable group $(\Gamma,\imath)$ with index $m$, and
let $F:=F_n$ be the $n$-th element of the canonically computable
F{\o}lner sequence $(F_n)_{n \geq 1}$. We compute the set $D:=F
\setminus  F g^{-1}$, which is seen as a subset of $\N$ with induced
ordering. Further, we compute the word $\widetilde
\omega_1:=\aast(\up)$. The increasing bijection $\imath_F: F \to \{1,2,\dots, \cntm{F} \}$  maps the subsets $F \cap F g^{-1} $ and $F g \cap F$ of $F$ to subsets $Y_1,Y_2$ of $\{ 1,2,\dots, \cntm{F}\}$. The right multiplication $R_g$ on $\Gamma$ is computable and restricts to a bijection from $F \cap  F g^{-1}$ to $F g \cap F$, so let $\widetilde{R_g}$ be the bijection making the diagram
\begin{displaymath}
    \xymatrix{
        {F \cap F g^{-1} } \ar[rr]^{\imath_{F}} \ar[d]_{R_g} & & Y_1 \ar[d]^{\widetilde{R_g}} \\
        {F g \cap F} \ar[rr]^{\imath_{F}}       & & Y_2  }
\end{displaymath}
commute. The output of $\adag$ is produced as follows. For  $x \in Y_1
\subseteq \{ 1, 2, \dots, \cntm{F}\}$ we set $\widetilde \omega_2(x):=
\widetilde \omega_1(\widetilde{R_g}(x))$, and the algorithm terminates
without producing output if $\widetilde{R_g}(x) > l(\widetilde
\omega_1)$ for some $x$. It is left to describe $\widetilde \omega_2$
on the remainder $Y_0:=\{ 1, 2, \dots, \cntm{F}\} \setminus Y_1$. We
let $\widetilde \omega_2|_{Y_0}:=\upsilon \circ \imath_{Y_0}$, where
$\imath_{Y_0}: Y_0 \to \{ 1,2,\dots, \card Y_0\}$ is an increasing
bijection. The algorithm prints nothing and terminates if $\card Y_0
\neq s$, otherwise the word $\widetilde \omega_2$ is printed.

Let $(g \cdot \omega)|_{F_n} \circ \imath_{F_n}^{-1}$ be the word on
$\{ 1,2,\dots, \cntm{F_n}\}$ that we want to encode, where $g \in
\Gamma$ has index $m$. Let $\up_n$ be an optimal description for
$\omega|_{F_n} \circ \imath_{F_n}^{-1}$ with respect to $\aast$. Let
$\upsilon$ be the word $(g \cdot \omega)|_{F_n \setminus F_n g^{-1}}
\circ \imath_{F_n \setminus F_n g^{-1}}^{-1}$. To encode the word $(g
\cdot \omega)|_{F_n} \circ \imath_{F_n}^{-1}$ using $\adag$, consider the program
\begin{equation*}
\widetilde \up_n:=\overline{\us} \del \uw \del \overline \un \del \overline \um \del \up_n,
\end{equation*}
where $\uw$ is the binary encoding of the $\Lambda$-word $\upsilon$ and $s = \cntm{F_n \setminus F_n g^{-1}}$. It is trivial to see that $\adag(\widetilde \up_n) = (g \cdot \omega)|_{F_n} \circ \imath_{F_n}^{-1}$.

The length of the program $\widetilde \up_n$ can be estimated by
\begin{equation*}
l(\widetilde \up_n) \leq \cntm{ F_n \setminus F_n g^{-1} } (\log{\card \Lambda}+1) + 2 \log\cntm{ F_n \setminus F_n g^{-1} } +c+ 2\log n + 2 \log m +l(\up_n),
\end{equation*}
where $c$ is some constant. By the definition of complexity of sections
\begin{equation*}
\widehat \uK(g \cdot \omega) = \limsup\limits_{n \to \infty} \frac{\kcc{\aast}{(g \cdot \omega)|_{F_n} \circ \imath_{F_n}^{-1}}}{\cntm{F_n}}.
\end{equation*}
Using that the optimal decompressor $\aast$ is not worse than $\adag$ (Equation \ref{eq.optdecomp}), we conclude that
\begin{align*}
&\kcc{\aast}{(g \cdot \omega)|_{F_n} \circ \imath_{F_n}^{-1}} \leq \kcc{\adag}{(g \cdot \omega)|_{F_n} \circ \imath_{F_n}^{-1}} + C \leq \\
&\leq \cntm{ F_n \setminus g^{-1} F_n } \cdot (\log{\card \Lambda}+1)+2 \log\cntm{ F_n \setminus F_n g^{-1} }+2\log n + l(\up_n) + C'
\end{align*}
for some constants $C,C'$ independent of $n$ and $\omega$. Taking the limits yields
\begin{equation*}
\limsup\limits_{n \to \infty} \frac{\kcc{\aast}{(g \cdot \omega)|_{F_n} \circ \imath_{F_n}^{-1}}}{\cntm{F_n}} \leq \limsup\limits_{n \to \infty} \frac{\kcc{\aast}{\omega|_{F_n} \circ \imath_{F_n}^{-1}}}{\cntm{F_n}}.
\end{equation*}
This completes the proof of the claim, and therefore the proof of the theorem.
\end{proof}

Of course, in the proof above we have not used that $\prX$ is closed. From now on we will omit explicit reference to the sequence $(F_n)_{n \geq 1}$ when talking about $\widehat \uK$. The proof of the following proposition is essentially similar to the original one in \cite{brudno1982}.

\begin{prop}
Let $(\Gamma, \imath)$ be a computable amenable group with a fixed canonically computable right F{\o}lner sequence $(F_n)_{n \geq 1}$ such that $\frac{\cntm{F_n}}{\log n} \to \infty$ as $n \to \infty$. Let $(\prX,\Gamma)$ be a subshift. For every $t \in \R_{\geq 0}$ the sets
\begin{align*}
&E_t:=\{ \omega \in \prX: \kcl{\omega} =t \},\\
&L_t:=\{ \omega \in \prX: \kcl{\omega} <t \},\\
&G_t:=\{ \omega \in \prX: \kcl{\omega} >t \}
\end{align*}
are measurable and shift-invariant.
\end{prop}
\begin{proof}
Invariance of the sets above follows from the previous proposition. We will now prove that the set $L_t$ is measurable, the measurability of other sets is proved in a similar manner. Observe that
\begin{align*}
L_t:=\{ \omega: \kcl{\omega} < t\} = \bigcup\limits_{k \geq 1} \bigcup\limits_{N \geq 1} \bigcap\limits_{n > N}\{ \omega: \kcc{\aast}{\omega|_{F_n} \circ \imath_{F_n}^{-1}}<(t- \frac 1 k)\cntm{F_n}\},
\end{align*}
and the sets $\{ \omega: \kcc{\aast}{\omega|_{F_n} \circ \imath_{F_n}^{-1}}<(t- \frac 1 k)\cntm{F_n} \}$ are measurable as finite unions of cylinder sets.
\end{proof}

We are now ready to prove the first inequality. The proof below is a slight adaption of the original one from \cite{brudno1982}.
\begin{thm}
\label{thm.brudnogeq}
Let $(\Gamma,\imath)$ be a computable group with a canonically computable tempered two-sided F{\o}lner sequence $(F_n)_{n \geq 1}$ such that $\frac{\cntm{F_n}}{\log n} \to \infty$. Let $(\prX,\Gamma)$ be a subshift on $\Gamma$, $\mu \in \Probm_{\Gamma}(\prX)$ be an ergodic $\Gamma$-invariant probability measure, and $\bfX=(\prX,\mu,\Gamma)$ be the associated measure-preserving system. Then $\kcl{\omega} \geq h(\bfX)$ for $\mu$-a.e. $\omega$.
\end{thm}
\begin{proof}
Suppose this is false, and let
\begin{equation*}
R:=\{ \omega: \kcl{\omega} < h(\bfX)\}
\end{equation*}
be the measurable set of words whose complexity is strictly smaller
than the entropy $h(\bfX)$. By the assumption, $\mu(R)>0$. The measure
$\mu$ is ergodic and the set $R$ is invariant, hence $\mu(R) = 1$. For
every $i \geq 1$ let
\begin{equation*}
R_i:=\lbrace \omega: \kcl{\omega} < h(\bfX) - \frac 1 i \rbrace,
\end{equation*}
then $R = \bigcup\limits_{i \geq 1} R_i$ and the sets $R_i$ are
measurable and invariant for all $i$. It follows that there exists an
index $i_0$ s.t. $\mu(R_{i_0}) = 1$. For every $l \geq 1$ define the set
\begin{equation*}
Q_{l}:=\{ \omega: \kcc{\aast}{\omega|_{F_i} \circ \imath_{F_i}^{-1}} < \left(h(\bfX) - \frac 1 {i_0}\right) \cntm{F_i} \text{ for all } i \geq l \},
\end{equation*}
then $Q_{l}$ is a measurable set for every $l \geq 1$ and $R_{i_0} =
\bigcup\limits_{l \geq 1} Q_{l}$. Let $1> \delta > 0$ be fixed. The sequence of sets $(Q_{l})_{l
  \geq 1}$ is monotone increasing, hence there is $l_0$ such that for all $l \geq l_0$ we have $\mu(Q_{l})> 1 - \delta$.

Let $\varepsilon<\min(\frac 1 {i_0}, 1 - \delta)$ be positive. Let $n_0:=n_0(\varepsilon) \geq l_0$ s.t. for all $n \geq n_0$ we have the decomposition $\prX = A_n \sqcup B_n$, where $\mu(B_n)<\varepsilon$ and for all $\omega \in A_n$ the inequality
\begin{equation}
\label{eq.smbest}
2^{-\cntm{F_n}(h(\bfX) + \varepsilon)} \leq \mu(\alpha_{\Lambda}^{F_n}(\omega)) \leq 2^{-\cntm{F_n}(h(\bfX) - \varepsilon)}
\end{equation}
holds. Such $n_0$ exists due to Corollary \ref{cor.smb}. For every $l \geq n_0$, we partition the sets $Q_{l}$ in the following way:
\begin{align*}
Q_{l}^A &:=Q_{l}\cap A_l; \\
Q_{l}^B &:=Q_{l}\cap B_l.
\end{align*}
It is clear that for every $l \geq n_0$
\begin{align*}
\mu(Q_{l}^B) &<\varepsilon; \\
 \mu(Q_{l}^A) \geq 1 - &\delta - \varepsilon>0.
\end{align*}
By the definition of the set $Q_{l}^A$, for all $l \geq n_0$ and all $\omega \in Q_{l}^A$ we
have 
\begin{equation*}
\kcc{\aast}{\omega|_{F_l} \circ \imath_{F_l}^{-1}} \leq \cntm{F_l}
(h(\bfX)-\frac 1 {i_0}).
\end{equation*}
This allows, for every $l \geq n_0$, to estimate the cardinality of the set of all restrictions of words in $Q_{l}^A$ to $F_l$ as
\begin{equation*}
\left|\{ \omega|_{F_l}: \omega \in  Q_{l}^A \} \right| \leq 2^{|F_l| (h(\bfX)-\frac 1 {i_0})+1},
\end{equation*}
which can be seen by counting all binary programs of length at most $|F_l| (h(\bfX)-\frac 1 {i_0})$. Combining this with the Equation \ref{eq.smbest}, we deduce that
\begin{equation*}
\mu(Q_{l}^A) \leq 2^{|F_l| (h(\bfX)-\frac 1 {i_0})+1} \cdot 2^{-|F_l|(h(\bfX) - \varepsilon)} \leq 2^{|F_l|(\varepsilon - \frac 1 {i_0})+1}.
\end{equation*}
This implies that $\mu(Q_{l}^A) \to 0$ as $l \to \infty$, since $|F_l|
\to \infty$ and $\varepsilon - \frac 1 {i_0}<0$. This contradicts to the estimate 
\begin{equation*}
\mu(Q_{l}^A) \geq 1-\delta-\varepsilon
\end{equation*}
for all $l \geq n_0$ above.
\end{proof}

\subsection{Part B}
\label{ss.partb}
In this part of the proof we shall derive the other inequality, which is technically more difficult to prove. We begin with a preliminary lemma.
\label{ss.partb}

\begin{lemma}
\label{lem.brudno}
Let $\bfX=(\prX,\mu,\Gamma)$ be an ergodic measure-preserving system,
where the discrete group $\Gamma$ admits a regular symmetric F{\o}lner
monotiling $([F_n,\calZ_n])_{n \geq 1}$. Let $(\beta_k)_{k \geq 1}$ be
a sequence of finite partitions of $\prX$, where $\beta_k = \{
B_1^k,B_2^k,\dots,B_{M_k}^k\}$ for all $k \geq 1$. For all $k
\geq 1$, $h \in \Gamma$, $m \in \{ 1,2,\dots,M_k\}$ let
\begin{equation}
\pi_{n,m}^{k, h}(\omega):=\avg{g \in F_n \cap \calZ_k} \indi{B_m^k}((gh) \cdot \omega)
\end{equation}
and
\begin{equation}
\tilde \pi_{n,m}^{k, h}(\omega):=\avg{g \in \kintl{F_k}{F_n}  \cap \kintr{F_k^{-1}}{F_n} \cap \calZ_k} \indi{B_m^k}((gh) \cdot \omega).
\end{equation}
Then the following assertions hold:
\begin{aufzi}
\item For $\mu$-a.e. $\omega \in \prX$ the limit
\begin{equation*}
\pi_{m}^{k,h}(\omega):=\lim\limits_{n \to \infty} \pi_{n,m}^{k,h}(\omega) = \lim\limits_{n \to \infty} \tilde \pi_{n,m}^{k,h}(\omega)
\end{equation*}
exists for all $k \geq 1$, $m \in \{1,2,\dots,M_k \}$ and $h \in \Gamma$.

\item For $\mu$-a.e. $\omega \in \prX$ and all $k \geq 1$ there exists $h := h_k(\omega) \in F_k^{-1}$ such that
\begin{equation*}
-\sum\limits_{m=1}^{M_k} \pi_{m}^{k,h}(\omega) \log \pi_{m}^{k,h}(\omega) \leq h_{\mu}(\beta_k).
\end{equation*}
\end{aufzi}

\end{lemma}
\begin{proof}
The first assertion follows from the definition of a regular F{\o}lner
monotiling, Theorem \ref{thm.wghtd} and countability of $\Gamma$. 

For the second assertion, observe that for $\mu$-a.e. $\omega$, all $k
\geq 1$ and all $m
\in \{1,2,\dots,M_k \}$ 
\begin{equation*}
\frac{1}{\cntm{F_k}} \sum\limits_{h \in F_k^{-1}} \pi_{m}^{k,h}(\omega) = \lim\limits_{n \to \infty} \avg{g \in F_n} \indi{B_m^k}(g \cdot \omega),
\end{equation*}
since,  for every $k \geq 1$, $[\calZ_k, F_k^{-1}]$ is a right
monotiling, 
\begin{equation*}
(\kintl{F_k}{F_n}
\cap \kintr{F_k^{-1}}{F_n} \cap \calZ_k) F_k^{-1} \subseteq F_n
\end{equation*}
for all $n \geq 1$ and 
\begin{equation*}
\frac{\cntm{(\kintl{F_k}{F_n}
\cap \kintr{F_k^{-1}}{F_n} \cap \calZ_k) F_k^{-1}}}{\cntm{F_n}} \to 1
\end{equation*}
as $n \to \infty$.

Using ergodicity of $\bfX$, we deduce that for $\mu$-a.e. $\omega$,
all $k \geq 1$ and all $m
\in \{1,2,\dots,M_k \}$ 
\begin{equation*}
\frac{1}{\cntm{F_k}} \sum\limits_{h \in F_k^{-1}}
\pi_{m}^{k,h}(\omega) = \lim\limits_{n \to \infty} \avg{g \in F_n} \indi{B_m^k}(g \cdot \omega) = \mu(B_m^k),
\end{equation*}
and the second assertion follows by the concavity of the entropy.
\end{proof}

We are now ready to prove the converse inequality. The proof is based on essentially the same idea of `frequency encoding', but the technical details differ quite a bit.
\begin{thm}
\label{thm.brudnoleq}
Let $(\Gamma,\imath)$ be a computable group with a fixed computable regular symmetric F{\o}lner monotiling  $([F_n, \calZ_n])_{n \geq 1}$. Let $(\prX,\Gamma)$ be a subshift on $\Gamma$, $\mu \in \Probm_{\Gamma}(\prX)$ be an ergodic measure and $\bfX=(\prX,\mu,\Gamma)$ be the associated measure-preserving system. Then $\kcl{\omega} \leq h(\bfX)$ for $\mu$-a.e. $\omega$.
\end{thm}
\begin{proof}
We will now describe a decompressor $\aex$ that will be used to encode restrictions of the words in $\prX$. The decompressor $\aex$ is defined on the domain of the programs of the form
\begin{equation}
\up:=\overline \us \del \overline \ut \del \overline \uf_1 \del \dots \overline \uf_L \ddel \overline \ur \del \uw \del \underline \uN.
\end{equation}
Here $\overline \us, \overline \ut,\overline \ur$ are doubling
encodings of some natural numbers $s,t,r$. Words $\overline \uf_1,
\dots, \overline \uf_L$, where we require that $L = (\card
\Lambda)^{\cntm{F_s}}$, are doubling encodings of nonnegative integers
$f_1,\dots,f_L$. The word $\uw$ encodes a $\Lambda$-word $\upsilon$ of
length $r$. The word $\underline \uN$ encodes\footnote{We stress that we use a binary encoding here and not a doubling encoding.} a natural number $N$. Observe that this interpretation is not ambiguous. Let $$\{ \widetilde \omega_1,\widetilde \omega_2,\dots, \widetilde \omega_L \}$$ be the list of all $\Lambda$-words of length $\cntm{F_s}$ ordered lexicographically.

The decompressor $\aex$ works as follows. From $s$ and $t$ compute the finite subsets \begin{equation*}
F_s, F_t, \kintl{F_s}{F_t} \cap \kintr{F_s^{-1}}{F_t}
\end{equation*}
of $\N$. Compute the finite set $$I_{s,t}:=\kintl{F_s}{F_t} \cap \kintr{F_s^{-1}}{F_t} \cap \calZ_s$$ of centers of monotiling $[F_s,\calZ_s]$. Next, for every $h \in I_{s,t}$ compute the tile $T_h:=F_s h \subseteq F_t$ centered at $h$. We compute the union $$\Delta_{s,t}:=\bigcup\limits_{h \in I_{s,t}} T_h \subseteq F_t$$ of all such tiles.

We will construct a $\Lambda$-word $\sigma$ on the set $F_t$, then $\widetilde \sigma:=\sigma \circ \imath_{F_t}^{-1}$ yields a word on $\{ 1,2,\dots,\cntm{F_t}\}$. The word $\sigma$ is computed as follows. First, we describe how to compute the restriction $\sigma|_{\Delta_{s,t}}$. For every $h \in I_{s,t}$ the word $\sigma \circ \imath_{T_h}^{-1}$ is a word on $\{ 1,2,\dots, \cntm{F_s}\}$, hence it coincides with one of the words
\begin{equation*}
\widetilde \omega_1, \widetilde \omega_2, \dots, \widetilde \omega_L
\end{equation*}
introduced above. We require that the word $\widetilde \omega_i$
occurs exactly $f_i$ times for every $i\in \{ 1,\dots,L \}$. This
amounts to saying that the word $\sigma|_{\Delta_{s,t}}$ has the \emph{collection of frequencies} $f_1,f_2,\dots,f_L$. Of course, this does not determine $\sigma|_{\Delta_{s,t}}$ uniquely, but only up to a certain permutation. Let $\calF_{\Lambda,\up}$ be the set of all $\Lambda$-words on $\Delta_{s,t}$ having collection of frequencies $f_1,f_2,\dots,f_L$. If $\sum\limits_{j=1}^L f_j \neq \cntm{I_{s,t}}$ the algorithm terminates and yields no output, otherwise $\calF_{\Lambda,\up}$ is nonempty. The set $\calF_{\Lambda,\up}$ is ordered lexicographically (recall that $\Delta_{s,t}$ is a subset of $\N$). It is clear that
\begin{equation}
\card \calF_{\Lambda,\up} = \frac{\cntm{I_{s,t}}!}{f_1! f_2! \dots f_L!}
\end{equation}
Thus to encode $\sigma|_{\Delta_{s,t}}$ it suffices to give the index
$N_{\calF_{\Lambda,\up}}(\sigma|_{\Delta_{s,t}})$ of
$\sigma|_{\Delta_{s,t}}$ in the set $\calF_{\Lambda,\up}$. We require
that $N_{\calF_{\Lambda,\up}}(\sigma|_{\Delta_{s,t}}) = N$, and this
together with the collection of frequencies $f_1,f_2,\dots,f_L$
determines the word $\sigma|_{\Delta_{s,t}}$ uniquely. If $N > \card \calF_{\Lambda,\up}$, the algorithm terminates without producing output.

Now we compute the restriction $\sigma_{F_t \setminus \Delta_{s,t}}$. Since $F_t \setminus \Delta_{s,t}$ is a finite subset of $\N$, we can simply list the values of $\sigma$ in the order they appear on $F_t \setminus \Delta_{s,t}$. That is, we require that
\begin{equation*}
\sigma|_{F_t \setminus \Delta_{s,t}} \circ \imath_{F_t \setminus \Delta_{s,t}}^{-1} = \upsilon,
\end{equation*}
and the algorithm terminates without producing output if $r \neq \card (F_t \setminus \Delta_{s,t})$.

For all $k \geq 1$, let $$\{ \widetilde \omega_1^k,\widetilde
\omega_2^k,\dots, \widetilde \omega_{M_k}^k \}$$ be the list of all
$\Lambda$-words of length $\cntm{F_k}$ ordered lexicographically. Here
$M_k = {(\card \Lambda)^{\cntm{F_k}}}$ for all $k$. For all $k \geq
1$ and $i\in \{ 1,\dots,M_k \}$ define the
cylinder sets 
\begin{equation*}
B_i^k:=\{ \omega \in \prX: \omega|_{F_k} \circ
\imath_{F_k}^{-1} = \widetilde \omega_i^k\},
\end{equation*}
and let $\beta_k:=\{ B_1^k,B_2^k, \dots, B_{M_k}^k\}$ be the corresponding
partition of $\prX$ into cylinder sets for every $k$. We apply Lemma
\ref{lem.brudno} to the system $\bfX=(\prX,\mu,\Gamma)$ and the sequence of partitions
$(\beta_k)_{k \geq 1}$. This yields a full measure subset $\prX_0
\subseteq \prX$ such that for all $\omega \in \prX_0$ and all $k \geq
1$ there is an element $h':=h_k(\omega) \in F_k^{-1}$ s.t.
\begin{equation}
\label{eq.brudnolem}
-\sum\limits_{m=1}^M \pi_{m}^{k,h'}(\omega) \log \pi_{m}^{k,h'}(\omega) \leq h_{\mu}(\beta_k).
\end{equation}
Let $\omega \in \prX_0$, $k \geq 1$ be arbitrary fixed and $h' :=
h_k(\omega) \in F_k^{-1}$ be the group element given by Lemma \ref{lem.brudno}. Because of the
shift-invariance of Kolmogorov complexity we have $\kcl{h' \cdot
  \omega} = \kcl{\omega}$. We will show that
\begin{equation*}
\kcl{\omega} = \kcl{h' \cdot \omega} \leq \frac{h_{\mu}(\beta_k)}{\cntm{F_k}}, 
\end{equation*}
then, since $h_{\mu}(\beta_k) = h_{\mu}(\alpha_{\Lambda}^{F_k})$ for
all $k \geq 1$, taking the limit as $k \to \infty$ completes the proof
of the theorem.

For the moment let $n$ be arbitrary fixed. Observe that for all $i \in \{
1,\dots,M_k \}$
\begin{equation*}
\widetilde \pi_{n,i}^{k,h'}(\omega)=\frac{1}{\cntm{I_{k,n}}}\sum\limits_{h \in I_{k,n}} \indi{B_i^k}(h \cdot (h'  \cdot \omega)),
\end{equation*}
i.e. $\cntm{I_{k,n}} \widetilde \pi_{n,i}^{k,h'}(\omega)$ equals the number of times the translates of the word $\widetilde \omega_i^k$ along the set $I_{k,n}$ appear in the word $(h' \cdot \omega)|_{\Delta_{k,n}}$. It follows by the definition of the algorithm $\aex$ that the following program describes the word $(h' \cdot \omega)|_{F_n} \circ \imath_{F_n}^{-1}$:
\begin{equation*}
\up:=\overline \uk \del \overline \un \del \overline \uf_1 \del \dots \overline \uf_{M_k} \ddel \overline \ur \del \uw \del \underline \uN
\end{equation*}
Here $\overline \uf_i$ is the doubling encoding of $\cntm{I_{k,n}}
\widetilde \pi_{n,i}^{k,h'}(\omega)$ for all $i\in \{ 1,\dots,M_k
\}$. The binary word $\uw$ encodes the word $\upsilon = (h' \cdot
\omega)|_{F_n \setminus \Delta_{k,n}} \circ \imath_{F_n \setminus
  \Delta_{k,n}}^{-1}$ of length $r$ and $\underline \uN$ encodes the index of $(h' \cdot \omega)|_{\Delta_{k,n}}$ in the set $\calF_{\Lambda,\up}$.

We will now estimate the length $l(\up)$ of the program $\up$
above. We begin by estimating the length of the word $\overline \uf_1
\del \dots \overline \uf_{M_k}$. Observe that
\begin{equation*}
f_j \leq \cntm{I_{k,n}} \leq \frac{\cntm{F_n}}{\cntm{F_k}} \ \text{ for every } j=1,\dots,M_k,
\end{equation*}
hence $l(\overline \uf_1 \del \dots \overline
\uf_{M_k})=o(\cntm{F_n})$. Next, we estimate the length
of the word $\uw$. Since $(F_n)_{n \geq 1}$ is a F{\o}lner sequence,
we conclude that $l(\uw)=o(\cntm{F_n})$. It is clear
that $l(\overline \un) \leq 2 \lfloor \log n \rfloor + 2=
o(\cntm{F_n})$, since $\frac{\cntm{F_n}}{\log n} \to \infty$. Finally,
we estimate $l(\underline \uN)$. Of course, $l(\underline \uN) \leq
\log \frac{\cntm{I_{k,n}}!}{f_1! f_2! \dots f_{M_k}!} + 1$. We use Stirling's approximation to deduce that
\begin{equation*}
\log \frac{\cntm{I_{k,n}}!}{f_1! f_2! \dots f_{M_k}!} \leq -\sum\limits_{j=1}^{M_k} f_j \log \frac{f_j}{\cntm{I_{k,n}}}+o(\cntm{F_n}).
\end{equation*}
Hence we can estimate the length of $\up$ by
\begin{align*}
l(\up) \leq o(\cntm{F_n}) -\sum\limits_{j=1}^{M_k} f_j \log \frac{f_j}{\cntm{I_{k,n}}}.
\end{align*}
Since $f_i = \cntm{I_{k,n}} \widetilde \pi_{n,i}^{k,h'}(\omega)$ for every $i=1,\dots,M_k$, we deduce that
\begin{align*}
l(\up) \leq o(\cntm{F_n}) -\cntm{I_{k,n}} \sum\limits_{j=1}^{M_k} \widetilde \pi_{n,j}^{k,h'}(\omega) \log \widetilde \pi_{n,j}^{k,h'}(\omega).
\end{align*}
Dividing both sides by $\cntm{F_n}$ and taking the limit as $n \to \infty$, we use Lemma \ref{lem.brudno} and Proposition \ref{prop.fmonot} to conclude that
\begin{equation*}
\limsup\limits_{n \to \infty} \frac{\kcc{\aex}{(h' \cdot \omega)|_{F_n} \circ \imath_{F_n}^{-1}}}{\cntm{F_n}} \leq \frac{h_{\mu}(\beta_k)}{\cntm{F_k}}.
\end{equation*}
By the optimality of $\aast$ we deduce that
\begin{equation*}
\kcl{h' \cdot \omega} = \limsup\limits_{n \to \infty} \frac{\kcc{\aast}{(h' \cdot \omega)|_{F_n} \circ \imath_{F_n}^{-1}}}{\cntm{F_n}} \leq \frac{h_{\mu}(\beta_k)}{\cntm{F_k}}
\end{equation*}
and the proof is complete.
\end{proof}

\printbibliography[]
\end{document}